\documentclass[journal]{IEEEtran}  
\usepackage{amsmath}
\usepackage{amssymb}
\usepackage{amsthm}
\usepackage{enumitem}
\usepackage{multicol}

\usepackage{graphics} 
\usepackage{epsfig} 
\usepackage{epstopdf}
\usepackage{mathptmx} 
\usepackage{times} 
\usepackage{amsmath} 
\usepackage{amssymb}  
\usepackage{subfigure}
\usepackage[utf8]{inputenc}
\usepackage[english]{babel}
\usepackage{fancyhdr}
\usepackage{dsfont}

\usepackage{color}

\newtheorem{thm}{Theorem}
\newtheorem{defn}[thm]{Definition}
\newtheorem{prop}[thm]{Proposition}
\newtheorem{lem}[thm]{Lemma}
\newtheorem{rem}[thm]{Remark}
\newtheorem{cor}[thm]{Corollary}

\newtheorem{ex}{Example}

\newcommand{\mcl}[1]{\mathcal{#1}}

\DeclareMathOperator{\Tr}{Trace}

\newcommand{\R}{\mathbb{R}}

\newcommand{\N}{\mathbb{N}}

\newcommand{\eps}{\varepsilon}

\title{\LARGE \bf
A Converse Sum of Squares Lyapunov Function for Outer Approximation of Minimal Attractor Sets of Nonlinear Systems
}

\author{Morgan Jones%
	\thanks{M. Jones is with the Department of Automatic Control and Systems Engineering,
	The University of Sheffield, Amy Johnson Building, Mappin Street, Sheffield, S1 3JD. e-mail: {\tt \small morgan.jones@sheffield.ac.uk} },
	Matthew M. Peet
	\thanks{M. Peet is with the School for the Engineering of Matter, Transport and Energy, Arizona State University, Tempe, AZ, 85298 USA. e-mail: {\tt \small mpeet@asu.edu } }
}

\begin{document}

\maketitle
\thispagestyle{plain}
\pagestyle{plain}


%

%

\begin{abstract}
	Many dynamical systems described by nonlinear ODEs are unstable. Their associated solutions do not converge towards an equilibrium point, but rather converge towards some invariant subset of the state space called an attractor set. For a given ODE, in general, the existence, shape and structure of the attractor sets of the ODE are unknown. Fortunately, the sublevel sets of Lyapunov functions can provide bounds on the attractor sets of ODEs. In this paper we propose a new Lyapunov characterization of attractor sets that is well suited to the problem of finding the minimal attractor set. We show our Lyapunov characterization is non-conservative even when restricted to Sum-of-Squares (SOS) Lyapunov functions. Given these results, we propose a SOS programming problem based on determinant maximization that yields an SOS Lyapunov function whose $1$-sublevel set has minimal volume, is an attractor set itself, and provides an optimal outer approximation of the minimal attractor set of the ODE. Several numerical examples are presented including the Lorenz attractor and Van-der-Pol oscillator.

\end{abstract}

\vspace{-0.4cm}
\section{Introduction} \label{sec: intro}
In this paper we consider nonlinear Ordinary Differential Equations (ODEs) of the form
\begin{equation} \label{ode}
\dot{x}(t)=f(x(t)), \quad x(0)=x_0.
\end{equation}
where $f: \R^n \to \R^n$ is the vector field and $x_0 \in \R^n$ is the initial condition. We denote the solution map (which exists and is continuous on $x \in X\subset \R^n$ when $f$ is Lipschitz continuous and $X$ is compact and invariant under $f$) of the ODE~\eqref{ode} by $\phi_f: X \times [0, \infty) \to \R^n$ where
\vspace{-0.2cm}
\begin{align*}
\frac{d}{dt} \phi_f(x,t) & = f(\phi_f(x,t)) \text{ for all } x \in X \text{ and } t \ge 0, \\
\phi_f(x,0) & =x \text{ for all } x \in X.
\end{align*}

 An ODE is asymptotically stable about some equilibrium point, {$x^*$}, if there exists some neighbourhood of the equilibrium, $\mathcal{N}({x^*})$, such that $\lim_{t \to \infty} \phi_f(x,t)={x^*}$ for any $x \in \mathcal{N}({x^*})$. Attractor sets generalize the notion of asymptotic stability, but where solutions tend towards a compact invariant subset of $\R^n$ (rather than being restricted to tend towards a single equilibrium point). Specifically, a compact set $A \subset \R^n$ is said to be an attractor set of the ODE~\eqref{ode} if for all $x \in A$ there exists $\eps>0$ such that $\lim_{t \to \infty} \inf_{y \in A}||y - \phi_f(z,t)||_2 =0$ for all $z \in \{y \in \R^n: ||x-y||_2<\eps\}$, and $x \in A$ implies $\phi_f(x,t) \in A$ for all $t \ge 0$. An attractor set is said to be minimal if there does not exists any other attractor sets contained within it.

Attractor sets provide information about the long term behavior of dynamical systems. The computation of attractor is used for design of secure private communications~\cite{cuomo1993synchronization,zhao2018observer}, the computation of Unstable Periodic Orbits (UPOs)~\cite{lakshmi2020finding}, and risk quantification of financial systems~\cite{gao2018ultimate}. Furthermore, identification of minimal attractor sets can be used to bound the domain of strange attractors and ``non-determinism'' in chaos theory~\cite{Lee_2016}. 

  It is well known that the sublevel sets of Lyapunov functions yield attractor sets~\cite{lin1996smooth}. A Lyapunov function of an ODE is any function that is positive and decreases along the solution map of the ODE. In~\cite{Li_2004,Yu_2005} quadratic Lyapunov functions were used to estimate bounds for Lorenz attractor. In~\cite{goluskin2020bounding} attractor sets are indirectly approximated by searching for  Sum-of-Squares (SOS) Lyapunov functions that provide bounds for $ \sup_{(x,t) \in \Omega \times [0, \infty) }\Phi(\phi_f(x,t))$, where $\Omega \subset \R^n$, $\Phi: \R^n \to \R$, and $\phi_f$ is the solution map to some ODE~\eqref{ode}. In~\cite{jones2018using} attractor sets approximated by using SOS to search for Lyapunov functions outside some handpicked set $D \subset \R^n$ that is known to contain the attractor set. In~\cite{schlosser2020converging,wang2012polynomial} an alternative SOS based method was proposed for attractor set approximation. Impressively, the method proposed in~\cite{schlosser2020converging} was shown to provide an arbitrarily close approximation of an attractor set with respect to the Lebesgue measure. However, the methods in~\cite{schlosser2020converging,wang2012polynomial} do not yield Lyapunov functions and hence any approximation found cannot be shown to also be an attractor set.

 The problem of computing attractor sets is related to the problem of certifying the asymptotic stability of equilibrium points of an ODE~\eqref{ode}; since certifying $A^*=\{0\}$ is an attractor set of an ODE~\eqref{ode} is equivalent to showing the asymptotic stability of the ODE~\eqref{ode} about $0 \in \R^n$. The use of SOS Lyapunov functions to certify the asymptotic stability of equilibrium points of an ODE~\eqref{ode} has been well treated in the literature~\cite{zheng2018computing,anderson2015advances,cunis2020sum,valmorbida2017region,awrejcewicz2021estimating,jones2020Arbitrarily,20.500.11850/461553}. 

SOS programming provides a computationally tractable method for searching for SOS Lyapunov functions and hence computing attractor sets of ODEs. However, it is currently unknown how conservative it is to restrict the search {of} Lyapunov functions to SOS polynomials. The goal of this paper is then to: 1) Propose a Lyapunov characterization of attractor sets that is well suited to the problem of approximating the minimal attractor. 2) Show that for a given ODE with, sufficiently smooth vector field, there exists a sequence of SOS Lyapunov functions that yield optimal outer set approximations of attractor sets of the ODE. Note that, an optimal outer set approximation of a set $A^* \subset \R^n$ is any set $A \subset \R^n$ such that $A^* \subseteq A$ and $D(A^*,A)$ is minimal, where $D$ is some set metric.

 Specifically, given an ODE~\eqref{ode}, we propose a new Lyapunov characterization of attractor sets. We show that if {$V: \R^n \to \R$} satisfies,
\vspace{-0.2cm} \begin{align} \label{introeqn: LF ineq}
\nabla V(x)^T f(x)  \le -(V(x)-1) & \text{ for all } x \in \Omega,\\ \label{introeqn: LF ineq 2}
\{x \in \Omega: V(x) \le 1  \} & \subseteq \Omega^\circ,\\ \label{introeqn: LF ineq 3}
\{x \in \Omega: V(x) \le 1  \} & \ne \emptyset,
\end{align}
{where $\Omega \subset \R^n$ is some compact set and $\Omega^\circ$ is the interior of $\Omega$, then the $1$-sublevel set of $V$ is an attractor set of the ODE~\eqref{ode}. To approximate the minimal attractor set of an ODE we then propose a sequence of $d$-degree optimization problem, each solved by a $d$-degree Sum-of-Square (SOS) polynomial function that satisfies Eqs.~\eqref{introeqn: LF ineq}, \eqref{introeqn: LF ineq 2} and~\eqref{introeqn: LF ineq 3}, and has minimal $1$-sublevel set. We show in Corollary~\ref{cor: convergence of SOS to attrcator} that the sequence of $d$-degree solutions to the optimization problem yield a sequence of $1$-sublevel sets that each contain the minimal attractor of the ODE~\eqref{ode}, are themselves attractor sets, and converge to the minimal attractor of the ODE~\eqref{ode} with respect to the volume metric. }

Our proposed optimization problem for optimal outer set approximations of minimal attractors is solved by finding the SOS polynomial Lyapunov function with minimal $1$-sublevel set volume. Unfortunately, there is no known closed expression for the volume of a sublevel set of a polynomial~\cite{lasserre2019volume}; making our optimization problem hard to solve. For SOS polynomials, $V= z_d(x)^T P z_d(x)$ where $P>0$, rather than minimizing the sublevel set volume of $V$ directly there exist several heuristics based on maximizing the eigenvalues of $P$. For instance in~\cite{Dabbene_2017} an optimization problem was proposed with $\Tr(P)$ objective function. Alternatively, $\log \det(P)$ functions have also been used as a metric for volume of $\{x \in \R^n: z_d(x)^T P z_d(x) \le 1  \}$, first being proposed in~\cite{Magnani_2005} and subsequently being used in the works of~\cite{Ahmadi_2017, jones2018using, jones2019using}. In this paper we also take a similar determinant maximizing approach and maximize $(\det(P))^{\frac{1}{n}}$ which is equivalent to maximizing $\log \det(P)$ but can be implemented on a larger array of SDP solvers~\cite{lofberg2004yalmip}.

 In order to establish the convergence of our proposed method for optimal outer approximations of minimal attractor sets we propose a new converse Lyapunov theorem. Specifically, given an attractor set we show that there exists a sequence of SOS Lyapunov functions each satisfying Eqs.~\eqref{introeqn: LF ineq}, \eqref{introeqn: LF ineq 2}, and~\eqref{introeqn: LF ineq 3}, and whose $1$-sublevel sets converge to the attractor set with respect to the volume metric. 
 
 Other important converse Lyapunov results concerning smooth Lyapunov functions include~\cite{lin1996smooth,teel2000smooth}; where it is shown that asymptotically stable nonlinear systems with sufficiently smooth vector fields admit smooth (but not necessarily SOS) Lyapunov functions that can certify the stability of the systems. In terms of SOS converse Lyapunov theory we mention~\cite{peet2010converse} that showed that if the system's solutions converge locally to an equilibrium point at an exponential rate then there always exists a SOS Lyapunov function that can certify this local exponential stability. However, for asymptotically stable systems whose solutions converge to an equilibrium point at a sub-exponential rate there may not exist SOS Lyapunov functions that can certify this stability, as shown by the counterexample presented in~\cite{ahmadi2018globally}.

Before proceeding, we note that there is no contradiction with the counterexample found in~\cite{ahmadi2018globally} and our proposed converse Lyapunov theorem (stated in Theorem~\ref{thm: existence of sos LF}). Although SOS Lyapunov functions cannot be used to certify the stability of equilibrium points in general (as proven by the counterexample from~\cite{ahmadi2018globally}), Theorem~\ref{thm: existence of sos LF} shows that SOS Lyapunov functions can be used to certify that arbitrarily small neighborhoods of equilibrium points are attractor sets. Hence SOS Lyapunov functions can certify the ``stability" of arbitrarily small neighborhoods of equilibrium points.



The rest of the paper is organized as follows. Notation is introduced in Section~\ref{sec: notation}. Attractor sets are defined in terms of solution maps of ODEs in Section~\ref{sec: Attractor sets are defines using solution maps}. A Lyapunov type theorem is proposed in Section~\ref{sec: LF} that provides sufficient conditions for a set to be an attractor set. In Section~\ref{sec:converse LF},  given an ODE, it is shown that there exists a sequence of SOS Lyapunov functions that yield a sequence of sublevel sets that converge to the minimal attractor set of the ODE. An SOS based algorithm for minimal attractor set approximation is then proposed in Section~\ref{sec: SOS problems for attractor set approx} and numerical examples are shown in Section~\ref{sec: numerical examples}. Finally our conclusion is given in Section~\ref{sec: conclusion}.



\vspace{-0.3cm}
\section{Notation} \label{sec: notation}
\vspace{-0.25cm}
\subsection{Set Metric Notation} \label{subsec: set metric notation}
For $A \subset \R^n$ we denote the indicator function by \\ {$\mathds{1}_A : \R^n \to \R$,} where $\mathds{1}_A(x) = \begin{cases}
& 1 \text{ if } x \in A\\
& 0 \text{ otherwise.}
\end{cases}$ For $B \subseteq \R^n$,  $\mu(B):=\int_{\R^n} \mathds{1}_B (x) dx$ is the Lebesgue measure of $B$. For sets $A,B \subset \R^n$, we denote the volume metric as $D_V(A,B)$, where $D_V(A,B):=\mu( (A/B) \cup (B/A) )$. We note that $D_V$ is a metric (Defn.~\ref{def:metric}), as shown in Lem.~\ref{lem: Dv is metric} (found in Appendix~\ref{sec: appendix 2}). For notation on the distance between a point and a set please see Section~\ref{subsec: set notation}.
\vspace{-0.2cm} \subsection{Euclidean Space Notation} \label{subsec: set notation}
We denote the power set of $\R^n$, the set of all subsets of $\R^n$, as $P(\R^n)=\{X:X\subset \R^n \}$. For two sets $A,B \in \R^n$ we denote $A/B= \{x \in A: x \notin B\}$. We denote the distance between a point $x \in \R^n$ and a set $A \subset \R^n$ by $D(x,A):=\inf_{y \in A}\{||x-y||_2 \}$. For $x \in \R^n$ we denote $||x||_p= \left( \sum_{i =1}^n x_i^p \right)^{\frac{1}{p}}$. For $\eta>0$ and $y \in \R^n$ we denote the set $B(y,\eta)= \{x\in \R^n : ||x-y||_2< \eta\}$. For $\eta>0$ and a set $A \subset \R^n$ we denote $B(A, \eta)=\cup_{x \in A} B(x, \eta)$. Let us denote bounded subsets of $\R^n$ by $\mcl B:=\{B \subset \R^n: \mu(B)<\infty\}$.  If $M$ is a subspace of a vector space $X$ we denote equivalence relation $\sim_M$ for $x,y \in X$ by $x \sim_M y$ if $x-y \in M$. We denote quotient space by $X \pmod M:=\{ \{y \in X: y \sim_M x \}: x \in X\}$. For a set $X \subset \R^n$ we say $x \in X$ is an interior point of $X$ if there exists $\eps>0$ such that $\{y \in \R^n: ||x-y||< \eps\}\subset X$. We denote the set of all interior points of $X$ by $X^\circ$. The point $x \in X$ is a limit point of $X$ if for all $ \eps>0$ there exists $ z \in \{y \in \R^n / \{x\}: ||x-y||<\eps\}$ such that $z \in X$; we denote the set of all limit points of $X$, called the closure of $X$, as $(X)^{cl}$. We say a set $X \subset \R^n$ is closed if $X=(X)^{cl}$. We say a set $X \subset \R^n$ is compact if it is closed and bounded. We denote the set of $n \times n$ {symmetric} matrices with strictly positive eigenvalues as $S_n^{++}$.
\vspace{-0.2cm} \subsection{Function Continuity and Smoothness Notation} \label{subsec: function continuity and smoothness notation}
 Let $C(\Omega, \Theta)$ be the set of continuous functions with domain $\Omega \subset \R^n$ and image $\Theta \subset \R^m$. We denote the set of locally and uniformly Lipschitz continuous functions on $\Theta_1 \text{ and }\Theta_2$ by $LocLip(\Theta_1,\Theta_2)$ and $Lip(\Theta_1,\Theta_2)$ respectively. For $\alpha \in \N^n$ we denote the partial derivative $D^\alpha f(x):= \Pi_{i=1}^{n} \frac{\partial^{\alpha_i} f}{\partial x_i^{\alpha_i}} (x)$ where by convention if $\alpha=[0,..,0]^T$ we denote $D^\alpha f(x):=f(x)$. We denote the set of $i$'th continuously differentiable functions by $C^i(\Omega,\Theta):=\{f \in C(\Omega,\Theta): D^\alpha f \in C(\Omega, \Theta) \text{ } \text{ for all } \alpha \in \N^n \text{ such that } \sum_{j=1}^{n} \alpha_j \le i\}$. For $V \in C^1(\R^n \times \R, \R)$ we denote $\nabla V:= (\frac{\partial V}{\partial x_1},....,\frac{\partial V}{\partial x_n})^T$. {We say $f: \Omega \to \R$ is such that $f \in L^1(\Omega,\R)$ if $||f||_{L^1(\Omega,\R)}:=\int_\Omega |f(x)| dx < \infty$.}.

\subsection{Polynomial Notation}
We denote the space of polynomials $p: \Omega \to \Theta$ by $\mcl P(\Omega,\Theta)$ and polynomials with degree at most $d \in \N$ by $\mcl{P}_d(\Omega,\Theta)$. We say $p \in \mcl{P}_{2d}(\R^n,\R)$ is Sum-of-Squares (SOS) if for $k \in \{1,...k\}\subset \N$ there exists $p_i \in \mcl{P}_{d}(\R^n,\R)$ such that $p(x) = \sum_{i=1}^{k} (p_i(x))^2$. We denote $\sum_{SOS}^d$ to be the set of SOS polynomials of at most degree $d \in \N$ and the set of all SOS polynomials as $\sum_{SOS}$. We denote $Z_d: \R^n \times \R \to \R^{\mcl N_d}$ as the vector of monomials of degree $d \in \N$ or less, where $\mcl N_d:= {d+n \choose d}$.


\section{Attractor Sets are Defined Using Solution Maps of Nonlinear ODEs} \label{sec: Attractor sets are defines using solution maps}
Consider a nonlinear Ordinary Differential Equation (ODE) of the form
\begin{equation} \label{eqn: ODE}
	\dot{x}(t) = f(x(t)), \quad x(0)=x_0\in \R^n, \quad t \in [0,\infty),
\end{equation}
where $f: \R^n \to \R^n$ is the vector field and $x_0 \in \R^n$ is the initial condition.

Given $X \subset \R^n$, $I \subset [0, \infty)$, and an ODE~\eqref{eqn: ODE} we say any function $\phi_f :X \times I \to \R^n$ satisfying
\begin{align} \label{soln map property}
	&\frac{\partial \phi_f(x,t)}{\partial t}= f(\phi_f(x,t)) \text{ for } (x,t) \in X \times I,\\ \nonumber
	& \phi_f(x,0)=x \text{ for } x \in X,\\ \nonumber
	& \phi_f(\phi_f(x,t),s)= \phi_f(x,t+s) \text{ for } x \in X \text{ } t,s \in I \text{ with } t+s \in I,
\end{align}
is a solution map of the ODE~\eqref{eqn: ODE} over $X \times I$. For simplicity throughout the paper we will assume there exists a unique solution map to the ODE~\eqref{eqn: ODE} over all $(x,t) \in \R^n \times [0,\infty)$. Note that the uniqueness and existence of a solution map sufficient for the purposes of this paper, such as for initial conditions inside some invariant set (like the Basin of Attraction of an attractor set given in Eq.~\eqref{eqn: ROA}) and for all $t \ge 0$, can be shown to hold under minor smoothness assumption on $f$, see~\cite{Khalil_1996}.

An important property of solution maps, we next recall in Lem.~\ref{lem: diff soln map}, is that they inherit the smoothness of their associated vector field. This smoothness property of solution maps is used in the proof of Prop.~\ref{prop: LF implies attractor set}.
\begin{lem}[Smoothness of the solution map. Page 149 \cite{Hirch_2004}]
	\label{lem: diff soln map}
	Consider $f \in C^1(\R^n,\R^n)$. Then if $\phi_f$ is a solution map (satisfying Eq.~\eqref{soln map property}) then $\phi_f \in C^1(\R^n \times \R, \R)$. 
\end{lem}

\subsection{Attractor Sets of Nonlinear ODEs}
A stable compact {attractor set} of the ODE~\eqref{eqn: ODE} is defined as follows.
\begin{defn} \label{defn: attractor set}
	We say that $A \subset \R^n$ is a \textbf{stable compact attractor set} of the ODE~\eqref{eqn: ODE}, defined by $f: \R^n \to \R^n$, if
	\begin{enumerate}
		\item $A$ is compact and nonempty ($A \ne \emptyset$).
		\item $A$ is a forward invariant set. That is if $\phi_f$ is a solution map of the ODE~\eqref{eqn: ODE} we have that,
		\begin{align} \label{set invariance}
	\phi_f(x,t) \in A \text{ for all } x \in A \text{ and } t \ge 0.
		\end{align}
		\item For each element of $A$ there is a neighbourhood of initial conditions for which the solution map asymptotically tends towards $A$. That is, for all $x \in A$ there exists $\delta>0$ such that for any $\eps>0$ there exists $T \ge 0$ for which
		\begin{align} \label{eqn: neighbourhood asym stab}
	D(\phi_f(y,t),A)< \eps \text{ for all } y \in B(x, \delta) \text{ and } t \ge T.
		\end{align}
\item {For all $\eps>0$ there exists $\delta>0$ such that if $x \in B(A, \delta)$ we have that $D(A, \phi_f(x,t))<\eps$ for all $t \ge 0$.}
	\end{enumerate}
	Furthermore, we say $A$ is a \textbf{minimal attractor set} if there does not exist any other attractor set, $B$, such that $B \subset A$, that is there exists $x \in A$ such that $x \notin B$ ($B$ is strictly contained in $A$).
\end{defn}
For simplicity we will often refer to {stable} compact attractor sets as attractor sets (leaving out the words {stable and} compact).

Note, in the case where $A\subset \R^n$ is a single point, that is $A= \{{x^*}\}$, the condition given in Eq.~\eqref{eqn: neighbourhood asym stab} reduces to the classical {conditions} of asymptotic stability of the equilibrium point ${x^*} \in \R^n$. That is, the condition  given in Eq.~\eqref{eqn: neighbourhood asym stab} reduces to requiring the existence of $\delta>0$ such that $ \lim_{t \to \infty}||\phi_f(x,t) - {x^*}||_2 =0$ for all $x \in B({x^*}, \delta)$, {along with the following stability condition: for all $\eps>0$ there exists a $\delta>0$  such that for all $t > 0$ we have that $\phi_f(B(\delta,x^*),t) \subset B(x^*,\eps)$.} 


Each attractor set of the ODE~\eqref{eqn: ODE} has an associated set of initial conditions for which solution maps initialized at these initial conditions converge towards the attractor set as $t \to \infty$. We call this set the {basin of attraction} of the attractor set and define it next.
\begin{defn} \label{defn: BOA}
Given an attractor set $A \subset \R^n$ of the ODE~\eqref{eqn: ODE} (defined by $f: \R^n \to \R^n$) we define the \textbf{basin of attraction} of $A$ as
	\begin{align} \label{eqn: ROA}
		BOA_f(A):=\left \{x \in \R^n: \lim_{t \to \infty} D(A,\phi_f(x,t))=0\right\}.
	\end{align}
\end{defn}
In the special case when the minimal attractor set is a single point the attractor set is commonly referred to as an equilibrium point and its associated basin of attraction is referred to as the region of attraction. However, although this special case is important for stability analysis, in general attractor sets can take more complicated structures such as limit cycles and in dimensions three and above (chaotic) ``strange attractors".


%
%
%

\section{A Lyapunov Approach to Finding and Certifying Minimal Attractor Sets} \label{sec: LF}

In this section, we propose a new Lyapunov characterization of attractor sets. To explain the motivation for this new characterization, consider a typical Lyapunov characterization of attractor sets, as given in~{\cite{bhatia2006dynamical}}.

\begin{thm}[Page~143 in~\cite{bhatia2006dynamical}] \label{thm: bhati LF compact attractor set}
{Consider {$f \in C^1(\R^n, \R^n)$}. A non-empty compact set $A \subset \R^n$ is an attractor set (Defn.~\ref{defn: attractor set}) of the ODE~\eqref{eqn: ODE}, defined by vector field $f$, if and only if there exists $V \in C(\R^n,\R)$ and $\Omega \subset \R^n$ such that
\begin{align} \label{e:1}
& A \subseteq \Omega^\circ. \\ \label{e:2}
& V(x)=0 \text{ if } x \in A \text{ and }  V(x)>0 \text{ if } x \in \Omega/ A.\\ \label{e:3}
&V(\phi_f(x,t))<V(x) \text{ for all } x \in \Omega / A \\ \nonumber 
& \hspace{1cm} \text{ and } t \in \{s \in (0,\infty): \phi_f(x,q) \in \Omega \text{ for all } q \in [0,s]  \}.
\end{align} }
\end{thm}

{Theorem~\ref{thm: bhati LF compact attractor set}} defines a method for certifying that a set $A \subset \R^n$ is an attractor set by searching for a Lyapunov function valid for $A$ -- an optimization problem with decision variable $V$. However, this formulation is not well-suited to the problem of \textit{finding} \textbf{minimal} attractor sets - a bilinear problem wherein both the attractor set $A$ and Lyapunov function $V$ are (unknown) decision variables. To resolve this problem, we propose Prop.~\ref{prop: LF implies attractor set}, wherein the proposed attractor set is defined as the 1-sublevel set of some function and hence there is only a single decision variable. {This decision variable can be thought of as a perturbed Lyapunov function. That is if $V$ is a Lyapunov function satisfying $\nabla V(x)^T f(x) < - V(x)$ then the perturbed Lyapunov function $\tilde{V}(x)= V(x)+1$ satisfies Eq.~\eqref{eqn: LF ineq} from Prop.~\ref{prop: LF implies attractor set}. }
	
	 {This perturbed Lyapunov function is no longer required to be zero over the attractor set (as in Eq.~\eqref{e:2}) making it more suitable to be searched for using polynomial optimization, since the only polynomial, $p$, that is zero over an open set is the zero polynomial $p \equiv 0$. Thus there may not exist a polynomial that satisfies Eqs.~\eqref{e:1}, \eqref{e:2} and~\eqref{e:3} from Theorem~\ref{thm: bhati LF compact attractor set}. On the other-hand, later in Section~\ref{sec: SOS problems for attractor set approx}, we will show that our proposed perturbed Lyapunov formulation, given in Prop~\ref{prop: LF implies attractor set}, allows us to combine the problems of certification and volume minimization of the attractor set using SOS programming and determinant maximization.}

\begin{prop} \label{prop: LF implies attractor set}
	Consider {$f \in C^1(\R^n, \R^n)$}. Suppose there exists $V \in C^1(\R^n, [0,\infty))$ such that
	\begin{align} \label{eqn: LF ineq}
\nabla V(x)^T f(x) \le - ( V(x) & - 1)  \text{ for all } x \in \Omega,\\ \label{eqn: S0 inside intererior}
\{x \in \Omega : V(x) \le 1\} & \subseteq \Omega^\circ,\\ \label{eqn: S0 non-empty}
\{x \in \Omega : V(x) \le 1\} & \ne \emptyset,
	\end{align}
	where $\Omega \subset \R^n$ is a compact set. Then $\{x \in \Omega : V(x) \le 1\}$ is an attractor set (Defn.~\ref{defn: attractor set}) to the ODE~\eqref{eqn: ODE} defined by $f$. \end{prop}
Note that the Lyapunov function $V$ in Prop.~\ref{prop: LF implies attractor set} is not required to be positive semidefinite. However, later in Section~\ref{sec: SOS problems for attractor set approx} we will include a positivity constraint on $V$ -- allowing us to minimize the volume of the $1$-sublevel set. 


\begin{proof}[Proof of Proposition~\ref{prop: LF implies attractor set}] 
{Suppose there exists $V$ that satisfies Eqs.~\eqref{eqn: LF ineq}, \eqref{eqn: S0 inside intererior}, and~\eqref{eqn: S0 non-empty}. We will now construct a compact set $A \subset \R^n$ and a function $\tilde{V} \in C^1(\R^n,\R)$ that satisfies Eqs.~\eqref{e:1}, \eqref{e:2} and~\eqref{e:3} from Theorem~\ref{thm: bhati LF compact attractor set}, hence showing $A$ is an attractor set. Let $A:=\{x \in \Omega : V(x) \le 1\}$. We first show that $A$ is compact.  Since $V$ is continuous it follows that $A:=\{x \in \Omega : V(x) \le 1  \}$ is closed by Lemma~\ref{lem: sublevel set is closed}. Moreover, $A$ is bounded since $A \subseteq \Omega^\circ$ and $\Omega$ is bounded. Since $A \subset \R^n$ is closed and bounded it follows that $A$ is a compact set. }

{Now, consider $\rho \in C^\infty(\R^n,\R)$ defined as
\begin{align*}
	\rho(x):= \begin{cases}
		e^{\frac{-1}{(1-x)^2}} \text{ if } x >1\\
			0 \text{ otherwise}
	\end{cases}.
\end{align*}
Define $\tilde{V} \in C^1(\R^n,\R)$ as
\begin{align*}
	\tilde{V}(x):=\rho(V(x)),
\end{align*}
where $V$ satisfies Eqs.~\eqref{eqn: LF ineq}, \eqref{eqn: S0 inside intererior}, and~\eqref{eqn: S0 non-empty}.}

{
Now, clearly  Eq.~\eqref{e:1} is trivially satisfied by $A:=\{x \in \Omega : V(x) \le 1\}$ using Eq.~\eqref{eqn: S0 inside intererior}. Also, $\tilde{V}$ satisfies Eq.~\eqref{e:2} since if $x \in A$ then $V(x) \le 1$ and hence $\tilde{V}(x)= \rho(V(x))=0$. On the other hand if $x \in \Omega/ A$  then $V(x)>1$ and hence $\tilde{V}(x)= \rho(V(x))>0$. Finally, we next show $\tilde{V}$ satisfies Eq.~\eqref{e:3}. If $x \in \Omega/ A$ then $V(x)>1$ and hence,
\begin{align} \label{e: LF int}
	 \nabla \tilde{V}(x)^Tf(x)  & =\rho'(V(x)) \nabla V(x)^T f(x) \\ \nonumber
	&= \frac{-2}{(1-V(x))^3} \rho(V(x)) \nabla V(x)^T f(x) \\ \nonumber
	& \le \frac{-2}{(1-V(x))^3} \rho(V(x))  (1-V(x)) <0.
\end{align} 
Eq.~\eqref{e: LF int} implies $\frac{d}{dt} \tilde{V}(\phi_f(x,t))<0$ $\text{ for all } x \in \Omega / A \text{ and } t \in \{s \in (0,\infty): \phi_f(x,q) \in \Omega  \text{ for all } q \in [0,s]\}.$ Note that if $t \in \{s \in (0,\infty): \phi_f(x,q) \in \Omega  \text{ for all } q \in [0,s]\}$ then $\phi_f(x,s) \in \Omega$ for all $s \in [0,t]$ and hence  $\frac{d}{dt} \tilde{V}(\phi_f(x,s))<0$ $\text{ for all } x \in \Omega / A \text{ and } s \in [0,t]$. Thus, $\int_0^t \tilde{V}(\phi_f(x,s)) ds<0$ and hence by the fundamental theorem of calculus it follows that $\tilde{V}$ satisfies Eq.~\eqref{e:3}.}

{Therefore $A=\{x \in \Omega : V(x) \le 1  \}$ is an attractor set by Thm.~\ref{thm: bhati LF compact attractor set}.}
\end{proof}
If $V$ and $\Omega$ satisfy Eqs.~\eqref{eqn: LF ineq}, \eqref{eqn: S0 inside intererior}, and~\eqref{eqn: S0 non-empty} (as in Prop.~\ref{prop: LF implies attractor set}) and $\{x \in \Omega : V(x) \le 1 + a\} \subseteq \Omega^\circ$ for some $a \ge 0$, then we next show that $\{x \in \Omega : V(x) \le 1 + a\}$ is a subset of the basin of attraction of the attractor set $\{x \in \Omega : V(x) \le 1 \}$.
\begin{cor}
	Consider {$f \in LocLip(\R^n, \R^n)$}. Suppose there exists $V \in C^1(\R^n, [0,\infty))$ and a compact set $\Omega \subset \R^n$ satisfying Eqs.~\eqref{eqn: LF ineq}, \eqref{eqn: S0 inside intererior}, and~\eqref{eqn: S0 non-empty} (as in Prop.~\ref{prop: LF implies attractor set}). Then, for any $a> 0$ such that $\{x\in \Omega : V(x) \le 1+a \} \subseteq \Omega^\circ$ it follows that $\{x\in \Omega: V(x) \le 1+a \} \subseteq BOA_f(\{x\in \Omega : V(x) \le 1 \})$.
\end{cor}
\begin{proof}
	{Throughout this proof we will use the following notation: $S_a:=\{x \in \Omega : V(x) \le 1 +a  \}$ where $a \ge 0$.}
	
	{\underline{\textbf{Proof $S_a \subseteq \Omega^\circ$ is an invariant set:}} We now prove that if $S_a \subseteq \Omega^\circ$, where $a \ge 0$, then $S_a$ is an invariant set. To see this, suppose for contradiction that there exists $y \in S_a$ and $T \ge 0$ such that $\phi_f(y,T) \notin S_a$. That is $V(\phi_f(y, 0)) \le 1 +a$ and $V(\phi_f(y, T)) > 1 +a$.  Now, since $V(\phi_f(y, \cdot))$ is {continuous} (since $V$ is continuous, $\phi_f$ is continuous by Lem.~\ref{lem: diff soln map}, and the composition of continuous functions is continuous) it follows by the intermediate value theorem that there exists $0 \le s_1 <s_2 \le T$ such that $V(\phi_f(y, s_1))=1 +a$ and $V(\phi_f(y, t))>1 +a$ for all $t \in (s_1,s_2]$. Thus $\phi_f(y, s_1) \in S_a \subseteq \Omega^\circ$ but $\phi_f(y, t) \notin S_a$ for all $t \in (s_1,s_2]$. Since $\Omega^\circ$ is open and $\phi_f(y, s_1) \in S_a \subseteq \Omega^\circ$ there exists $\eps>0$ such that $B(\phi_f(y, s_1), \eps) \subset \Omega^\circ$. Again, using the continuity of $V(\phi_f(y, \cdot))$ there exists $\delta>0$ such that $\phi_f(y, s_1 +s) \in B(\phi_f(y, s_1), \eps) \subseteq \Omega^\circ$ for all $s \in [0, \delta]$. Therefore, $V(\phi_f(y, t))>1+a$ and $\phi_f(y, t) \in \Omega^\circ$ for all $t \in (s_1,s_3]$, where $s_3:=\min\{s_2,s_1+\delta\}$. Applying the mean value theorem there exists $s_1 < c < s_3$ such that
	\begin{align} \label{contra 1}
	\frac{d}{dt} V(\phi_f(y,c)) & = \frac{V(\phi_f(y,s_3)) - V(\phi_f(y,s_1))}{s_3 - s_1} \\ \nonumber
	&> \frac{1+a-1-a}{s_3 - s_1}=0.
	\end{align}
	On the other hand since $\phi_f(y, t) \in \Omega^\circ$ for all $t \in (s_1,s_3]$ it follows that $\phi_f(y,c) \in \Omega^\circ$ and therefore Eq.~\eqref{eqn: LF ineq} can be applied to give
	\begin{align}\label{contra 2}
	\frac{d}{dt} V(\phi_f(y,c)) \le 1-V(\phi_f(y,c)) < 1-1-a =-a,
	\end{align}
	using the fact that $c \in (s_1,s_3)$ and $V(\phi_f(y, t))>1+a$ for all $t \in (s_1,s_3]$. }
	
	{Now Eqs.~\eqref{contra 1} and~\eqref{contra 2} contradict each other proving $S_a$ is invariant.}

	{\underline{\textbf{Proof $S_a \subseteq BOA_f(S_0)$:}} Since, $S_a \subset \Omega^\circ$ is invariant it follows that for any $x \in S_a$ we have that $\phi_f(x,t) \in \Omega^\circ$ for all $t \ge 0$. Thus, by Eq.~\eqref{eqn: LF ineq} we have that
		\begin{align*}
		\frac{d}{dt} V(\phi_f(x,t)) \le - &  (V(\phi_f(x,t)) - 1)  \\
		&\text{ for all } t \in [0,\infty).
		\end{align*}
		Now using Gronwall's inequality (Lem.~\ref{lem: gronwall}) and the fact $x \in S_a$ we have that
		\begin{align*}
		V(\phi_f(x,t)) - 1 \le e^{- t} & ( V(x) - 1) \le a e ^{- t} \\
		& \text{ for all } t \in [0,\infty).
		\end{align*}
		Therefore, it now follows for any $\eta>0$ that
		\begin{align} \label{222}
		\phi_f(x,t) \in S_\eta \text{ for all }  t \ge  \ln \left( \frac{a}{\eta} \right).
		\end{align} }

		{For any $\eps>0$ we will now construct $\eta>0$ such that $S_\eta \subseteq {B(S_0, \eps)}$ (recalling the notation $B(S_0, \eps)$ is defined in Sec~\ref{subsec: set notation}) implying that if $\phi_f(x,t) \in S_\eta$ for all $t \ge T$ then $D(S_0,\phi_f(x,t) ) < \eps$ for all $t \ge T$.}
		
		{First note that if $\Omega / B(S_0,\eps)= \emptyset$ then $\Omega \subseteq B(S_0,\eps)$ and we can trivially take $\eta= a$. Then by Eq.~\eqref{222} we have that $ \phi_f(x,t) \in S_a \subseteq \Omega^\circ \subseteq B(S_0,\eps)$ for all $t \ge 0$. Thus, $D(\phi_f(x,t), S_0)< \eps$ for all $t \ge 0$. }
		
		{Let us now consider the case $\Omega / B(S_0,\eps) \ne \emptyset$. Let $\eta \in  (0, b)$ where $b= \min\{\inf_{ z \in \Omega / B(S_0,\eps) }V(z) - 1,\frac{a}{2}\} $, where $\inf_{ z \in \Omega / B(S_0,\eps) }V(z)$ exists since $\Omega / B(S_0,\eps)$ is compact and $V$ is continuous. Note that $b>0$ since $a>0$ and $\inf_{ z \in \Omega / B(S_0,\eps) }V(z) - 1>0$ (because $\Omega / B(S_0,\eps)$ is compact so by the extreme value theorem there exists $z^* \in \Omega / B(S_0,\eps)$ such that $V(z^*)=\inf_{ z \in \Omega / B(S_0,\eps) }V(z)$ and since $z^* \notin S_0$ it follows that $V(z^*)> 1$).}

		 {We now claim that $S_\eta \subseteq B(S_0, \eps)$. First we note that $S_\eta \subseteq \Omega^\circ$ since $S_\eta \subset S_a$ and $S_a \subset \Omega^\circ$. Now suppose for contradiction that $S_\eta \nsubseteq B(S_0, \eps)$. Then there exists $w \in S_\eta \subseteq \Omega$ such that $w \notin B(S_0, \eps)$ implying $w \in \Omega / B(S_0, \eps)$. Now, $V(w) \le \eta + 1 < \inf_{ z \in \Omega / B(S_0,\eps) }\{V(z)\} \le V(w)$ implying $0<0$, providing a contradiction. }
		 
		 {Therefore we have shown that for any $x \in S_a$ and $\eps>0$ there exists $T \ge 0$ such that $D(\phi_f(x,t),S_0)<\eps$ implying $x \in BOA_f(S_0)$ and hence $S_a \subseteq BOA_f(S_0)$.}
\end{proof}

In Prop.~\ref{prop: LF implies attractor set} we have shown that if a function $V$ satisfies Eqs.~\eqref{eqn: LF ineq}, \eqref{eqn: S0 inside intererior} and~\eqref{eqn: S0 non-empty} then the $1$-sublevel set of $V$ is an attractor set of the ODE defined by $f$. In the next section we now prove that these Lyapunov characterizations of attractor sets are not conservative, even when $V$ is restricted to be an SOS polynomial.

\section{Converse Lyapunov Functions for Attractor Set Characterization} \label{sec:converse LF}
In the previous section, we have shown that if there exists a function $V$ which satisfies Eqs.~\eqref{eqn: LF ineq}, \eqref{eqn: S0 inside intererior} and~\eqref{eqn: S0 non-empty}, then the set $\{x \in \Omega : V(x)\le 1\}$ is an attractor set of the ODE defined by $f \in LocLip(\R^n, \R^n)$. In this section, we show that for \textbf{any} attractor set $A \subset \R^n$ and any $\epsilon>0$, there exists an SOS function $V$ which satisfies Eq.~\eqref{eqn: LF ineq} and for which $A \subset \{x \in \Omega : V(x)\le 1\} $ and $D_V(A, \{x \in \Omega : V(x)\le 1\})\le \epsilon $. This implies that the Lyapunov characterization of attractor sets in Section~\ref{sec: LF} is not conservative and furthermore, these conditions remain tight even when the Lyapunov functions are constrained to be SOS. In Section~\ref{sec: SOS problems for attractor set approx}, we will use this result to propose a sequence of SOS programming problems whose limit yields an attractor set which is arbitrarily close to the minimal attractor set.


{To begin, we quote a result on existence of smooth converse Lyapunov function from~\cite{teel2000smooth}.}

\begin{cor}[Cor.~2 in~\cite{teel2000smooth}] \label{cor: teel existsnece of smooth LF}
Consider {$f \in LocLip(\R^n, \R^n)$}. The set $A \subset \R^n$ is an attractor set to the ODE~\eqref{eqn: ODE} if and only if there exists $V \in C^\infty(BOA_f(A), \R)$ such that
\begin{enumerate}
	\item $V(x) \ge 0$ for all $x \in BOA_f(A)$ and $V(x) =0$ if and only if $x \in A$.
	\item $\nabla V(x) ^T f(x) \le - V(x)$ for all $x \in BOA_f(A)$.
\end{enumerate}
\end{cor}
{Next, in} Thm.~\ref{thm: existence of sos LF} we use Cor.~\ref{cor: teel existsnece of smooth LF} to show that for any given attractor set $A \subset \R^n$ there exists a sequence of Sum-of-Squares (SOS) polynomials, each satisfying {Eqs.~\eqref{eqn: LF ineq} and~\eqref{eqn: S0 inside intererior}}, each of whose 1-sublevel sets contain $A$, and whose 1-sublevel sets converge to $A$ (with respect to the volume metric). {Note that in order to show our SOS approximation, $P$, satisfies Eq.~\eqref{eqn: S0 inside intererior} we show $P(x)>1+\alpha$ for all $x \in \partial \Omega$ for some $\alpha>0$.}

To construct such an SOS function required in the proof of Thm.~\ref{thm: existence of sos LF}, we approximate the square root of the converse Lyapunov function $V$ (from Cor.~\ref{cor: teel existsnece of smooth LF}), perturbed by a positive constant, by a polynomial function. We then square this polynomial approximation to get an SOS approximation. Note that we perturb $V$ by a positive constant $\gamma>0$ since the image of $V$ includes $\{0\}$. Hence, without this perturbation, it follows that the square root of $V$ may not be differentiable, making it challenging to approximate it by a polynomial. To overcome this problem, rather than approximating $\sqrt{V(x)}$ we approximate $\sqrt{V(x) + \gamma}$.


\begin{thm} \label{thm: existence of sos LF}
	For $f \in LocLip(\R^n, \R^n)$, suppose $A \subset \R^n$ is an attractor set of the ODE~\eqref{eqn: ODE} defined by $f$ and let $\Omega$ be any compact set such that $A \subseteq \Omega^\circ$ and $\Omega \subset BOA_f(A)$. {Then there exists a sequence of polynomials, $\{P_d\}_{d \in \N} \subset \sum_{SOS}(\R^n, \R)$ with $P_d \in \sum_{SOS}^d(\R^n, \R)$ for all $d \in \N$, a positive definite scalar $\alpha>0$ and  an integer $N \in \N$ such that}
	\begin{enumerate}
		\item $\nabla P_d(x) ^T f(x) < -  (P_d(x) - 1)$ for all $x \in \Omega$ and $d \ge N$.
		\item $P_d(x) >1 + \alpha$ for all $x \in \partial \Omega$ and $d \ge N$.
			\item $A \subseteq \{x \in \Omega: P_d(x) \le 1 \}$ for all $d \ge N$.
			\item $\lim_{d \to \infty} D_V(A, \{x \in \Omega: P_d(x) \le 1 \})=0$ (recalling $D_V$ denotes the volume metric defined in Sec.~\ref{subsec: set metric notation}).
	\end{enumerate}
\end{thm}
\begin{proof}
	Let us suppose $A \subset \R^n$ is an attractor set to the ODE~\eqref{eqn: ODE}. By Cor.~\ref{cor: teel existsnece of smooth LF} there exists $W\in C^\infty(BOA_f(A), \R)$ such that
\begin{enumerate}
	\item $W(x) \ge 0$ for all $x \in BOA_f(A)$ and $W(x) =0$ if and only if $x \in A$.
	\item $\nabla W(x) ^T f(x) \le - W(x)$ for all $x \in BOA_f(A)$.
\end{enumerate}

{We next split the remainder of the proof of Theorem~\ref{thm: existence of sos LF} into the following parts. In Part 1 we perturb $W$ by a positive constant $\gamma>0$, defining $J(x):=W(x)=\gamma$, and approximate $H(x):=\sqrt{J(x)}$ by a $d$-degree polynomial function, $R_d$. In Part 2 we show that for a suitable correction term, $\sigma>0$, that the SOS polynomial, $G_d(x):=(R_d(x)-\sigma)^2$, can be made arbitrarily close to $J$ and satisfies a Lyapunov type inequality similar to Eq.~\eqref{eqn: LF ineq}. Finally, in Part 3, we show that the SOS function $P_d(x):=\frac{G_d(x)}{\gamma}$ satisfies the statement of Theorem~\ref{thm: existence of sos LF} .}

\underline{\textbf{Part 1 of the proof:}} Let $\gamma >0$ and consider $J(x):= W(x) + \gamma$. It trivially follows that since $A \subseteq \Omega \subset BOA_f(A)$ we have
\begin{align} \label{pfeq: J}
&\nabla J(x)^T f(x) \le -(J(x) - \gamma) \text{ for all } x \in BOA_f(A).\\ \label{pfeq: A = level set of J}
& A=\{x \in \Omega: J(x) \le \gamma \}.
\end{align}

Since $W(x) \ge 0$ it follows that  $J(x) \ge \gamma >0$. Therefore $H(x):=+\sqrt{J(x)}$ is differentiable. That is $H \in C^1(\R^n, \R)$, since the function $g(x):=\sqrt{x}$ is differentiable over $(0, \infty)$ and $J$ maps onto $(\gamma, \infty) \subset (0, \infty)$. Using the fact $H$ is differentiable and applying the chain rule we find that
\begin{align} \nonumber
||\nabla H(x)||_2 & = ||\nabla \sqrt{W(x) + \gamma}||_2 = \frac{1}{2\sqrt{W(x) + \gamma}} ||\nabla W(x)||_2\\ \label{pfeq: M3 bound}
& \le \frac{1}{2\sqrt{  \gamma}}||\nabla W(x)||_2 \le \frac{C}{2\sqrt{\gamma}} \text{ for all } x \in \Omega,
\end{align}
where $C:=\sup_{x \in \Omega}||\nabla W(x)||_2$. Note that the first inequality in Eq.~\eqref{pfeq: M3 bound} follows since $W(x) \ge 0$ for all $x \in \Omega$ implies $\frac{1}{\sqrt{W(x) + \gamma}} \le \frac{1}{\sqrt{\gamma}}$ for all $x \in \Omega$. 

Moreover, applying the chain rule, the inequality in Eq.~\eqref{pfeq: J}, and the fact that $\frac{1}{2 \sqrt{J(x)}} \ge 0$ for all $x \in \Omega$ we find that,
\begin{align*}
\nabla H(x)^T f(x) & = \nabla \sqrt{J(x)}^T f(x)= \frac{1}{2\sqrt{J(x)}}\nabla {J(x)}^T f(x)\\
& \le \frac{-1}{2\sqrt{J(x)}} (J(x) - \gamma) \text{ for all } x \in \Omega.
\end{align*}
It now follows that $H$ satisfies
\begin{align} \label{pfeq: H PDE}
2H(x)\nabla H(x)^T f(x) & \le - (H^2(x) - \gamma) \text{ for all } x \in \Omega.
\end{align}

We next approximate $H$ by a polynomial. Because in Part 2 of the proof we use this polynomial approximation of $H$ to construct an SOS polynomial approximation of $H(x)^2:=J(x)$, we require specific error bounds in our polynomial approximation of $H$. Specifically, let
 \vspace{-0.25cm} \begin{align} \label{pfeq:gamma}
\gamma> \frac{M_1C}{2}>0,
 \end{align}
where $M_1:=\sup_{x \in \Omega}||f(x)||_2$ and recalling $C:=\sup_{x \in \Omega}||\nabla W(x)||_2$. Note that $\gamma$ from Eq.~\eqref{pfeq:gamma} is a constant that only depends on the problems data ($f$ and $\Omega$).

 Also let,
 \vspace{-0.25cm} \begin{align}
& \eps>0,  \\ \label{pfeq:alpha}
& 0<\alpha< \frac{1}{\gamma}\min_{x \in \partial \Omega} W(x) \\ \label{pfeq:theta}
& 0<\theta< \min \left\{ \eps, (\mu(\Omega) +1)(\min_{x \in \partial \Omega} W(x) - \gamma \alpha)  \right\} \\ \label{pfeq: eps}
&0< \delta < \min \bigg\{\frac{ \sqrt{\gamma} -  M_1 M_3}{ M_1 M_2 + M_1 M_3 + M_2}, \frac{\sqrt{\gamma}}{M_2}  \bigg\},\\ \label{pfeq: sigma}
&0< \sigma < \min \bigg\{ \frac{2( \sqrt{\gamma} -  (M_1 M_2 + M_1 M_3 + M_2) \delta - M_1 M_3)}{(2M_1 +1 )\delta^2 + 2(1+M_1)\delta + 1  },\\ \nonumber
& \frac{2 (\sqrt{\gamma} \hspace{-0.05cm} - \hspace{-0.05cm} M_2\delta)}{(\delta+1)^2}, \frac{\sqrt{\theta}}{\sqrt{2(\mu(\Omega) \hspace{-0.05cm} + \hspace{-0.05cm} 1)}(\delta \hspace{-0.05cm} + \hspace{-0.05cm} 1)}, \frac{\theta}{4M_2(\delta \hspace{-0.05cm}+ \hspace{-0.05cm} 1)(\mu(\Omega) \hspace{-0.05cm} + \hspace{-0.05cm}1)}  \bigg\},
\end{align}
recalling $M_1:=\sup_{x \in \Omega}||f(x)||_2\ge0$ and where $M_2:=\sup_{x \in \Omega}|H(x)|\ge0$, and $M_3:=\sup_{x \in \Omega}||\nabla H(x)||_2\ge0$. Note that $\alpha>0$ since $\gamma>0$ and $\min_{\delta \in \partial \Omega} W(x) >0$ (since $A \subseteq \Omega ^\circ$ implies $A \cap \partial \Omega = \emptyset$ and $W(x)= 0$ iff $x \in A$). Also note that $\theta>0$ since $\eps>0$ and $\min_{x \in \partial \Omega} W(x) - \gamma \alpha$ by Eq.~\eqref{pfeq:alpha}. Moreover, $\delta>0$ since $\gamma> \frac{M_1C}{2}$ (by Eq.~\eqref{pfeq:gamma}) and $M_3 \le \frac{C}{2 \sqrt{\gamma}}$ (by Eq.~\eqref{pfeq: M3 bound}) implying that $ \sqrt{\gamma} - M_1 M_3>0$. Furthermore, $\sigma>0$ since $\delta<\frac{ \sqrt{\gamma} -  M_1 M_3}{ M_1 M_2 + M_1 M_3 + M_2}$ implying $2( \sqrt{\gamma} -  (M_1 M_2 + M_1 M_3 + M_2) \delta - M_1 M_3)>0$ and $\delta< \frac{\sqrt{\gamma}}{M_2}$ implying $2 (\sqrt{\gamma} -  M_2 \delta)>0$.

 Now, by Theorem~\ref{thm:Nachbin} there exists polynomials $\{R_d\}_{d \in \N} \subset \mcl P(\R^n, \R)$ and $N \in \N$ such that
\vspace{-0.25cm} \begin{align} \label{pfeq: H-R< eps sigma}
 |H(x) - R_d(x)| < \delta \sigma \text{ for all } d \ge N.\\ \label{pfeq: nabla H-R< eps sigma}
||\nabla H(x) - \nabla R_d(x)||_2 < \delta \sigma \text{ for all } d \ge N.
\end{align}

\underline{\textbf{Part 2 of the proof:}} In this part of the proof we show that there exists $\gamma>0$, $\alpha>0$, $N \in \N$, and $\{G_d\}_{d \in \N} \subset \sum_{SOS}$ such that
\begin{align} \label{pfeq: P is LF}
	& \nabla {G}_d(x) ^T f(x) < -  ({G}_d(x) - \gamma) \text{ for all } x \in \Omega \text{ and } d > N.\\ \label{pfeq: less than J}
	&{G}_d(x) \le J(x) \text{ for all } x \in \Omega \text{ and } d \ge N.\\  \label{pfeq: greater than gamma on boundary}
	& G_d(x) \ge \gamma(1+ \alpha) \text{ for all } x \in \partial \Omega \text{ and } d \ge N \\ \label{pfeq: convergence in L1}
	& \lim_{d \to \infty} ||{G}_d - J||_{L^1(\Omega, \R)}=0.
\end{align}


Before we proceed we note that showing Eq.~\eqref{pfeq: convergence in L1} is equivalent to showing that there exists $\{G_d\}_{d \in \N} \subset \sum_{SOS}$ such that for any $\eps>0$ there exists $N \in \N$ for which the following holds
\vspace{-0.25cm}
\begin{align} \label{pfeq:P close in L1 to J}
	||G_d-J||_{L^1(\Omega, \R)}< \eps \text{ for all } d \ge N.
\end{align}

Now, since $R_d \in \mcl P(\R^n, \R)$ (defined by Eqs~\eqref{pfeq: H-R< eps sigma} and~\eqref{pfeq: nabla H-R< eps sigma}) it follows that $G_d(x):=(R_d(x)- \sigma)^2$ is a SOS polynomial, that is $ G_d \in \sum_{SOS}$ for all $d \in \N$. 

We next show that $G_d$ satisfies Eq.~\eqref{pfeq: P is LF}. Recalling $\gamma> \frac{M_1C}{2}$ (by Eq.~\eqref{pfeq:gamma}), $M_1:=\sup_{x \in \Omega} ||f(x)||_2$ and $C:= \sup_{x \in \Omega} ||\nabla W(x)||_2$, it follows that
 \vspace{-0.1cm} \begin{align} \nonumber
&\nabla G_d(x)^T f(x) + (G_d(x) - \gamma) \\ \nonumber
& = (\nabla(R_d(x) - \sigma)^2)^T f(x)  + ((R_d(x) - \sigma)^2 - \gamma)\\ \nonumber
& =  2 (R_d(x) -\sigma) \nabla R_d(x)^T f(x) + (R_d^2(x) - \gamma) - 2\sigma R_d(x) + \sigma^2\\ \nonumber
& \le  2 R_d(x) \nabla R_d(x)^T \hspace{-0.1cm} f(x) - 2 H(x) \nabla H(x)^T \hspace{-0.1cm} f(x)  + (R_d^2(x) - H^2(x)) \\ \nonumber
& \hspace{1cm} - 2 \sigma  R_d(x) + \sigma^2 -2\sigma \nabla R_d(x)^T f(x) \\ \nonumber
& = 2(R_d(x) - H(x)) \nabla R_d(x)^T \hspace{-0.1cm} f(x)  + 2H(x)\nabla (R_d  -  H)(x) ^T \hspace{-0.1cm} f(x)\\ \nonumber
& \qquad  + (R_d(x) - H(x))(R_d(x) + H(x)) + 2 \sigma (H(x) - R_d(x)) \\ \nonumber
& \qquad - 2\sigma\nabla (R_d  -  H)(x) ^Tf(x) -2 \sigma \nabla H(x)^Tf(x) \\ \nonumber
& \qquad - 2\sigma H(x) + \sigma^2 \\ \nonumber
& \le 2 |R_d(x) - H(x)|||\nabla R_d(x)||_2||f(x)||_2 \\ \nonumber
& \qquad + 2H(x) ||\nabla R_d(x) - \nabla H(x)||_2 ||f(x)||_2 \\ \nonumber
& \qquad + |R_d(x) - H(x)|(|R_d(x)| + H(x)) + 2 \sigma |H(x) - R_d(x)| \\ \nonumber
 & \qquad + 2 \sigma ||\nabla (R_d  -  H)(x) ||_2||f(x)||_2
+ 2 \sigma ||\nabla H(x) ||_2 ||f(x)||_2 \\ \nonumber
& \qquad - 2 \sigma \sqrt{\gamma} + \sigma^2\\ \nonumber
& \le 2 \delta \sigma M_1 (||\nabla (R_d  -  H)(x)||_2 + ||\nabla H(x)||_2 ) + 2 M_1M_2 \delta \sigma \\ \nonumber
& \qquad + \delta \sigma ( |R_d(x) - H(x)| + H(x) + M_2) + 2 \delta \sigma^2 + 2M_1 \delta \sigma^2 \\ \nonumber
& \qquad + 2 M_1 M_3 \sigma - 2 \sigma \sqrt{\gamma} + \sigma^2\\ \nonumber
& \le 2M_1 \delta^2 \sigma^2 + 2M_1 M_3 \delta \sigma + 2 M_1 M_2 \delta \sigma + \delta^2 \sigma^2 + 2M_2 \delta \sigma \\ \nonumber
& \qquad + 2 \delta \sigma^2 + 2 M_1 \delta \sigma^2 + 2 M_1 M_3 \sigma - 2 \sqrt{\gamma} \sigma + \sigma^2\\ \nonumber
& = \sigma\bigg( ((2M_1 +1 )\delta^2 + 2(1+M_1)\delta + 1 )\sigma \\ \nonumber
& \qquad +2(( M_1 M_2  + M_1 M_3  +M_2)\delta + M_1 M_3 - \sqrt{\gamma} ) \bigg) \\ \label{pfeq: P is a LF}
& <0 \text{ for all } x \in \Omega \text{ and } d \ge N.
\end{align}
Where all the equalities in Eq.~\eqref{pfeq: P is a LF} follow from rearranging terms or adding and subtracting terms. The first inequality in Eq.~\eqref{pfeq: P is a LF} follows by applying the inequality in Eq.~\eqref{pfeq: H PDE}. The second inequality in Eq.~\eqref{pfeq: P is a LF} follows by the triangle inequality and the Cauchy Swarz in inequality. The third and fourth inequalities in Eq.~\eqref{pfeq: P is a LF} follows by Eqs.~\eqref{pfeq: H-R< eps sigma} and~\eqref{pfeq: nabla H-R< eps sigma}. Finally, the last inequality (the fifth inequality) in Eq.~\eqref{pfeq: P is a LF} follows by Eq.~\eqref{pfeq: sigma}.

We now show $G_d$ satisfies Eq.~\eqref{pfeq: less than J}.
\begin{align} \nonumber
& G_d(x) - J(x)= (R_d(x) - \sigma)^2 - H(x)^2\\ \nonumber
& =R_d(x)^2 - 2 \sigma R_d(x) + \sigma^2 - H(x)^2\\ \nonumber
& = (R_d(x) - H(x))(R_d(x) + H(x)) + 2 \sigma (H(x) - R_d(x)) \\ \nonumber
& \qquad - 2 \sigma H(x) + \sigma^2\\ \nonumber
& \le \delta \sigma (\delta \sigma + 2M_2) + 2\delta \sigma^2 - 2 \sigma \sqrt{\gamma} + \sigma^2\\ \label{pfeq: P is less than J}
& = \sigma \bigg( (\delta^2 + 2 \delta + 1)\sigma + 2M_2 \delta - 2 \sqrt{\gamma}      \bigg) \\ \nonumber
&< 0 \text{ for all } x \in \Omega \text{ and } d \ge N.
\end{align}
Where the first inequality in Eq.~\eqref{pfeq: P is less than J} follows using Eq.~\eqref{pfeq: H-R< eps sigma} and the fact $H(x) \ge \sqrt{\gamma}$ for all $x \in \Omega$. The second inequality in Eq.~\eqref{pfeq: P is less than J} follows by Eq.~\eqref{pfeq: sigma} $\left(\sigma<\frac{2 \sqrt{\gamma} - 2M_2\delta}{(\delta+1)^2} \right)$.

We now show $G_d$ satisfies Eq.~\eqref{pfeq: greater than gamma on boundary}.
\begin{align} \nonumber
& J(x)-G_d(x) = J(x) -(R_d(x) - \sigma)^2 \\ \nonumber
& =H(x)^2 - R_d(x)^2 + 2 \sigma R_d(x) - \sigma^2 \\ \nonumber
& = (H(x) - R_d(x))(H(x) + R_d(x)) + 2 \sigma (R_d(x) -H(x)) \\ \nonumber
& \qquad + 2 \sigma H(x) - \sigma^2 \\ \nonumber
& \le \delta \sigma (2 M_2 + \delta \sigma) + 2\delta \sigma^2 + 2 \sigma M_2 - \sigma^2\\  \nonumber
& =  (\delta^2 + 2 \delta -1) \sigma^2 +2M_2 (\delta+1) \sigma \\ \nonumber
& \le (\delta^2 + 2 \delta +1) \sigma^2 +2M_2 (\delta+1) \sigma\\
& < \frac{\theta}{\mu(\Omega)+1} \text{ for all } x \in \Omega \text{ and } d \ge N. \label{pfeq: J is less than P}
\end{align}
Where the first inequality in Eq.~\eqref{pfeq: J is less than P} follows using Eq.~\eqref{pfeq: H-R< eps sigma}. The second inequality in Eq.~\eqref{pfeq: J is less than P} follows by $\sigma>0$ and $-1<1$. The third inequality in Eq.~\eqref{pfeq: J is less than P} follows by Eq.~\eqref{pfeq: sigma} ($\sigma<\frac{\sqrt{\theta}}{\sqrt{2 (\mu(\Omega)+1)}(\delta+1)}$ and $\sigma< \frac{\theta}{4M_2(\delta+1)(\mu(\Omega)+1)}$). Now by rearranging Eq.~\eqref{pfeq: J is less than P} and using the fact that $J(x):= W(x) + \gamma$ we have that,
\begin{align} \label{pfeq: G greater than gam on boundary}
G_d(x) & > J(x) - \frac{\theta}{\mu(\Omega)+1}\\ \nonumber
& =W(x) + \gamma - \frac{\theta}{\mu(\Omega)+1} \\ \nonumber
& =\gamma +  \frac{ (\mu(\Omega)+1)(\min_{x \in \partial \Omega}W(x)) - \theta}{\mu(\Omega)+1} \\ \nonumber
& > \gamma(1+ \alpha) \text{ for all } x \in \partial \Omega \text{ and } d \ge N.
\end{align}
Where the first inequality in Eq.~\eqref{pfeq: G greater than gam on boundary} follows by Eq.~\eqref{pfeq: J is less than P} and the second inequality follows by Eq.~\eqref{pfeq:theta}. Hence Eq.~\eqref{pfeq: G greater than gam on boundary} shows Eq.~\eqref{pfeq: greater than gamma on boundary} holds.

We now show $G_d$ satisfies Eq.~\eqref{pfeq: convergence in L1} by showing $G_d$ satisfies Eq.~\eqref{pfeq:P close in L1 to J}. By Eqs.~\eqref{pfeq: P is less than J} and~\eqref{pfeq: J is less than P} and the fact that $\theta< \eps$ (Eq.~\eqref{pfeq:theta}), it follows that
\begin{align}
|G_d(x) -J(x)|< \frac{\eps}{\mu(\Omega)+1} \text{ for all } x \in \Omega \text{ and } d \ge N,
\end{align}
and thus
\begin{align*}
||G_d - J||_{L^1(\Omega, \R)} \le \sup_{x \in \Omega}|G_d(x) -J(x)| \mu(\Omega) < \eps \text{ and } d \ge N.
\end{align*}
Therefore Eq.~\eqref{pfeq: convergence in L1} holds.

\underline{\textbf{Part 3 of the proof:}} Let $P_d(x):= \frac{G_d(x)}{\gamma}$. Recall that $\gamma>0$ from Eq.~\eqref{pfeq:gamma} is a constant that only depends on the problems data ($f$ and $\Omega$) and not $d \in \N$. Therefore, $\lim_{d \to \infty} P_d= \frac{1}{\gamma} \lim_{d \to \infty} G_d$. Moreover, it follows that $\{P_d\}_{d \in \N} \subset \sum_{SOS}$ since $\{G_d\}_{d \in \N} \subset \sum_{SOS}$ and $\gamma>0$ (by Eq.~\eqref{pfeq:gamma}). Furthermore, it follows by Eqs.~\eqref{pfeq: P is LF}, \eqref{pfeq: less than J}, \eqref{pfeq: greater than gamma on boundary}, and~\eqref{pfeq: convergence in L1} that
\begin{align} \label{pfeq: 2 P is LF}
& \nabla {P}_d(x) ^T f(x) < -  ({P}_d(x) - 1) \text{ for all } x \in \Omega \text{ and } d \ge N,\\ \label{pfeq: 2 less than J}
&{P}_d(x) \le \tilde{J}(x) \text{ for all } x \in \Omega \text{ and } d \ge N,\\ \label{pfeq: P greater than 1 on boundary}
& P_d(x)>1 + \alpha \text{ for all } x \in \partial \Omega \text{ and } d \ge N \\ \label{pfeq: 2 convergence in L1}
& \lim_{d \to \infty} ||{P}_d - \tilde{J}||_{L^1(\Omega, \R)}=0,
\end{align}
where $\tilde{J}(x)=\frac{J(x)}{\gamma}$.

We now argue that Theorem~\ref{thm: existence of sos LF} is proven. Eq.~\eqref{pfeq: A = level set of J} implies that $A=\{x \in \Omega: \tilde{J}(x)\le 1 \}$. Then Eq.~\eqref{pfeq: 2 less than J} implies that $A \subseteq \{x \in \Omega: P_d(x) \le 1 \}$ for all $d \ge N$. Moreover, Eqs.~\eqref{pfeq: 2 less than J}~and~\eqref{pfeq: 2 convergence in L1} together with Theorem~\ref{prop: close in L1 implies close in V norm} (found in Appendix~\ref{sec: appendix 2}) imply that $\lim_{d \to \infty} D_V(\{x \in \Omega: \tilde{J}(x)\le 1 \}, \{x \in \Omega: P_d(x) \le 1 \})=0$, implying $\lim_{d \to \infty} D_V(A, \{x \in \Omega: P_d(x) \le 1 \})=0$ (since $A=\{x \in \Omega: \tilde{J}(x)\le 1 \}$).
\end{proof}

\begin{rem}
Theorem~\ref{thm: existence of sos LF} shows that for any attractor set, $A$, there exists an SOS polynomial, $P$, that satisfies the Lyapunov conditions (Eqs.~\eqref{eqn: LF ineq}, \eqref{eqn: S0 inside intererior} and~\eqref{eqn: S0 non-empty}) of Prop.~\ref{prop: LF implies attractor set} and that has a $1$-sublevel set arbitrarily close to the attractor set with respect to the volume metric. Note that $P$ satisfies Eq.~\eqref{eqn: LF ineq} directly from the statement of Theorem~\ref{thm: existence of sos LF}. Also note that $P$ satisfies Eq.~\eqref{eqn: S0 inside intererior} since by Theorem~\ref{thm: existence of sos LF} we have that $P(x)>1$ for all $\delta \Omega$. Also note that Eq.~~\eqref{eqn: S0 non-empty} since by Theorem~\ref{thm: existence of sos LF} we have that $A \subseteq \{x \in \Omega : P(x) \le 1\}$ and since $A \ne \emptyset$ it follows that $\{x \in \Omega : P(x) \le 1\} \ne \emptyset$.
\end{rem}

Theorem~\ref{thm: existence of sos LF} shows that the Lyapunov characterization of attractor sets proposed in Section~\ref{sec: LF} is not conservative and that this non conservatism is retained even if the Lyapunov functions are constrained to be SOS. However, in order to apply the results of Sections~\ref{sec: LF} and~\ref{sec:converse LF} to compute outer approximations of minimal attractors, we require an algorithm which can enforce the Lyapunov inequality conditions of Prop.~\ref{prop: LF implies attractor set} while minimizing the volume of the 1-sublevel set of the Lyapunov function. In the following section propose such an algorithm based on convex optimization and SOS programming.

\section{A Family of SOS Problems for Minimal Attractor Set Approximation} \label{sec: SOS problems for attractor set approx}
In Section~\ref{sec: LF}, we proposed a Lyapunov characterization of attractor sets for a given ODE defined by a vector field $f$. In Section~\ref{sec:converse LF}, we showed that this characterization is not conservative even if the Lyapunov functions are constrained to be SOS. Given these two results, we may now formulate a polynomial optimization characterization of the minimal attractor set $A^* \subset \Omega$ of a given ODE defined by a vector field, $f$. The following optimization problem enforces the Lyapunov conditions of Prop.\ref{prop: LF implies attractor set} while minimizing the distance between the minimal attractor $A^*$ and the $1$-sublevel set of the Lyapunov function:
\begin{align} \label{opt: approx minimal A}
&\inf_{J \in \mcl F } D_V(A^*,\{x \in \Omega: J(x) \le 1 \})\\ \nonumber
& \text{such that } \nabla J(x) ^T f(x) \le -  (J(x) - 1) \text{ for all } x \in \Omega,\\ \nonumber
& \hspace{1.5cm} \{x \in \Omega : J(x) \le 1\}  \subseteq \Omega^\circ,\\ \nonumber
& \hspace{1.5cm} \{x \in \Omega : J(x) \le 1\}  \ne \emptyset,
\end{align}
where $\mcl F$ is some set of functions which we may take to be the set of SOS polynomials.

In Subsection~\ref{subsec: SOS constraints}, we will propose a SOS programming approach to solving Optimization Problem~\eqref{opt: approx minimal A}. Specifically, in Subsection~\ref{subsec: SOS constraints}, we propose a sequence of quasi-SOS programming problems, each involving volume minimization, and whose limit yields the minimal attractor set of the ODE defined by $f$. However, the SOS constraints in Subsection~\ref{subsec: SOS constraints} do not enforce $\{x \in \Omega : J(x) \le 1\}  \ne \emptyset$ - thus reducing the computational complexity of the algorithm. We show that it is not necesary to enforce this constraint because as we will show next in Subsection~\ref{subsec:reduce opt}, by selecting $\Omega$ sufficiently large and enforcing $\nabla J(x) ^T f(x) \le -  (J(x) - 1) \text{ for all } x \in \Omega$ it follows that $J$ automatically satisfies $\{x \in \Omega : J(x) \le 1\}  \ne \emptyset$. Moreover, unlike Opt.~\eqref{opt: approx minimal A}, the objective function of our proposed quasi-SOS programming problem will not involve the unknown set $A^*$. This is because, as we will next show in Subsection~\ref{subsec:reduce opt}, for sufficiently large $\Omega$ and $J$ such that $\nabla J(x) ^T f(x) \le -  (J(x) - 1) \text{ for all } x \in \Omega$ it follows $A^* \subseteq \{x\in \Omega: J(x) \le 1  \}$ (the $1$-sublevel set of $J$ contains the minimal attractor set). We later use result in Subsection~\ref{subsec: SOS constraints} to eliminate $A^*$ from the objective function.
 
  Note in addition, we will show in Subsection~\ref{subsec:futher simp} that if sublevel set volume is minimized and $\Omega$ is sufficiently large, then we may likewise eliminate the constraint $\{x \in \Omega : J(x) \le 1\}  \subseteq \Omega^\circ$-- thus further reducing computational complexity of the SOS programming problem.


\subsection{A Reduced Form of Optimization Problem~\eqref{opt: approx minimal A}} \label{subsec:reduce opt}

In Prop.~\ref{prop: LF implies attractor set} we have proposed a Lyapunov characterization of attractor sets. We have shown that if $V$ satisfies Eqs.~\eqref{eqn: LF ineq}, \eqref{eqn: S0 inside intererior} and~\eqref{eqn: S0 non-empty} then the $1$-sublevel set of $V$ is an attractor set of the ODE defined by the vector field $f$. In Eq.~\eqref{opt: approx minimal A} we have proposed an optimization problem that searches over functions $J$ that satisfy Eqs.~\eqref{eqn: LF ineq}, \eqref{eqn: S0 inside intererior} and~\eqref{eqn: S0 non-empty} while minimizing the distance between the $1$-sublevel set of $J$ and the minimal attractor set of the ODE defined by $f$.

Later, in Subsection~\ref{subsec: SOS constraints} we will propose an SOS programming problem for solving Opt.~\eqref{opt: approx minimal A} that searches for a $J$ that satisfies Eqs.~\eqref{eqn: LF ineq} and \eqref{eqn: S0 inside intererior} while minimizing the volume of the $1$-sublevel set of $J$, but does not directly enforce Eq.~\eqref{eqn: S0 non-empty} - instead choosing $\Omega$ to be sufficiently large. Fortunately, as we will show next in Lemma~\ref{lem: S0 is nonempty} that if $\Omega$ is chosen sufficiently large such that $A \subseteq \Omega$, for some attractor set $A$ of the ODE, then any continuous $V$ satisfying Eq.~\eqref{eqn: LF ineq} automatically satisfies Eq.~\eqref{eqn: S0 non-empty}. Lemma~\ref{lem: S0 is nonempty} then shows that if $\Omega$ contains the minimal attractor and $V$ satisfies Eqs.~\eqref{eqn: LF ineq} and \eqref{eqn: S0 inside intererior}, then the 1-sublevel set of $V$ is an attractor set.  

\begin{lem} \label{lem: S0 is nonempty}
	Consider {$f \in C^1(\R^n, \R^n)$}. Suppose there exists {an attractor set} (Defn.~\ref{defn: attractor set}) $A \subset \R^n$ of the ODE~\eqref{eqn: ODE} defined by $f$, $V \in C^1(\R^n, [0,\infty))$, and a compact set $\Omega \subset \R^n$ such that
\begin{align} \label{1}
\nabla V(x)^T f(x) &\le - ( V(x) - 1)  \text{ for all } x \in \Omega,\\ \label{2}
A \subseteq \Omega,
\end{align}
then $\{x \in \Omega: V(x) \le 1 \} \ne \emptyset$.
\end{lem}
\begin{proof} In order to prove $\{x \in \Omega: V(x) \le 1 \} \ne \emptyset$ we show that $A \cap \{x \in \Omega: V(x) \le 1 \} \ne \emptyset$.

 Suppose for contradiction that $A \cap \{x \in \Omega: V(x) \le 1 \}=\emptyset$. Then $V(y)>1$ for all $y \in A$. Since $A \subseteq \Omega$ (by Eq.~\eqref{2}) is an attractor set it is an invariant set. Therefore, $\phi_f(y,t) \in A \subseteq \Omega$ for all $t \ge 0$ and thus by Eq.~\eqref{1} it follows that
\begin{align*}
\frac{d}{dt} V(\phi_f(y,t)) \le - &  (V(\phi_f(y,t)) - 1) \text{ for all } (y,t) \in A \times [0,\infty).
\end{align*}
Then, using Gronwall's inequality (Lem.~\ref{lem: gronwall}) we have that
\begin{align} \label{3}
V(\phi_f(y,t)) - 1 \le e^{- t} & ( V(y) - 1)  \text{ for all } (y,t) \in A \times [0,\infty).
\end{align}

Let $c:=\inf_{t \ge 0} \{V(\phi_f(y,t))-1 \}$. We will now argue that $c>0$. Using the fact that $\phi_f(y,t) \in A$ for all $t \ge 0$ it follows that $c=\inf_{t \ge 0} \{V(\phi_f(y,t))-1 \} \ge \inf_{z \in A} \{V(z)-1 \}$. Then, since $V$ is continuous and $A$ is compact it follows by the extreme value theorem that there exists $z^* \in A$ such that $V(z^*)-1=\inf_{z \in A} \{V(z)-1 \} \ge c$. Since we have assumed $A \cap \{x \in \Omega: V(x) \le 1 \}=\emptyset$ it follows that if $z^* \in A$ then $z^* \notin  \{x \in \Omega: V(x) \le 1 \}$ and hence $c \ge V(z^*)-1>0$.

Now, by Eq.~\eqref{3} and since $c>0$ it follows that $0<ce^t  \le V(y) -1$ for all $t \ge 0$ and $y \in A$ implying that $V$ is unbounded over $A$, contradicting the continuity of $V$. Therefore it follows that $A \cap \{x \in \Omega: V(x) \le 1 \} \ne \emptyset$ and hence $\{x \in \Omega: V(x) \le 1 \} \ne \emptyset$.
\end{proof}
Later in Subsection~\ref{subsec: SOS constraints} we will propose an optimization problem that has an objective function independent of the unknown set $A^*$ (unlike Opt.~\ref{opt: approx minimal A}). In order to formulate this optimization problem we require $A^* \subseteq \{x \in \Omega: V(x) \le 1 \}$. Next, we show that if $\Omega$ contains a neighborhood of the minimal attractor, $A^*$, and $V$ satisfies Eqs.~\eqref{eqn: LF ineq}, then the 1-sublevel set of $V$ contains the minimal attractor set.

\begin{lem} \label{lem: S0 contains minimal attractor}
	Consider $f \in C^1(\R^n, \R^n)$. Suppose $A^* \subset \R^n$ is the minimal attractor set (Defn.~\ref{defn: attractor set}) of the ODE~\eqref{eqn: ODE} defined by $f$, $V \in C^1(\R^n, [0,\infty))$, $\sigma>0$, and a compact set $\Omega \subset \R^n$ such that
	\begin{align} \label{Lf}
	\nabla V(x)^T f(x) &\le - ( V(x) - 1)  \text{ for all } x \in \Omega,\\ \label{subset}
	{B(A^*,\sigma) \subseteq \Omega},
	\end{align}
	then $A^* \subseteq \{x \in \Omega: V(x) \le 1 \}$.
\end{lem}
\begin{proof}{
	To show $A^* \subseteq \{x \in \Omega: V(x) \le 1 \}$ we will show $A^* \cap \{x \in \Omega: V(x) \le 1 \}$ is an attractor set. Then if $A^* \nsubseteq \{x \in \Omega: V(x) \le 1 \}$ it follows that $A^* \cap \{x \in \Omega: V(x) \le 1 \} \subset A^* $, that is there exists an attractor set that is a strict subset of $A^*$, contradicting the fact that $A^*$ is the minimal attractor set.  }
	
	To show $A^* \cap \{x \in \Omega: V(x) \le 1 \}$ is an attractor set we will split the remainder of the proof into three parts, showing $A^* \cap \{x \in \Omega: V(x) \le 1 \}$ satisfies the three properties of attractor sets in Defn.~\ref{defn: attractor set}.
	
	\underline{\textbf{Proof $A^* \cap \{x \in \Omega: V(x) \le 1 \}$ is nonempty and compact:}} By the proof of Lemma~\ref{lem: S0 is nonempty} it follows that $A^* \cap \{x \in \Omega: V(x) \le 1 \} \ne \emptyset$. Moreover, since $A^*$ is compact and $\Omega$ is compact, implying $\{x \in \Omega: V(x) \le 1 \} \subseteq \Omega$ is compact, it follows $\{x \in \Omega: V(x) \le 1 \}$ is compact.
	
	\underline{\textbf{Proof $A^* \cap \{x \in \Omega: V(x) \le 1 \}$ is invariant:}} Let $y \in A^* \cap \{x \in \Omega: V(x) \le 1 \}$ then $y \in A^*$ and $y \in \{x \in \Omega: V(x) \le 1 \}$. Since $A^*$ is an attractor set it is invariant and therefore $\phi_f(y,t) \in A^*$ for all $t \ge 0$. In order to prove $A^* \cap \{x \in \Omega: V(x) \le 1 \}$ is invariant we must also show $\phi_f(y,t) \in \{x \in \Omega: V(x) \le 1 \}$ for all $t \ge 0$. For contradiction suppose there exists $T>0$ such that $\phi_f(y,t) \notin \{x \in \Omega: V(x) \le 1 \}$. That is, $V(\phi_f(y,T))>1$. 
	
	Using the fact $A^*$ is invariant and applying the Granwall Bellman Lemma to Eq.~\eqref{Lf} we get,
	\[V(\phi_f(y,t)) - 1 \le e^{- t}  ( V(y) - 1)  \text{ for all } (y,t) \in A^* \times [0,\infty).\]
Hence, if $V(\phi_f(y,T))>1$ we get that that $0<V(\phi_f(y,T)) - 1 \le e^T (V(y) -1)$ implying $V(y) >1$ contradicting the fact that $y \in \{x \in \Omega: V(x) \le 1 \}$. Thus we have shown that if $y \in A^* \cap \{x \in \Omega: V(x) \le 1 \}$ then $\phi_f(x,t) \in A^*$ for all $t \ge 0$ and $\phi_f(x,t) \in  \{x \in \Omega: V(x) \le 1 \}$ implying $\phi_f(x,t) \in A^* \cap \{x \in \Omega: V(x) \le 1 \}$ for all $t \ge 0$, proving $A^* \cap \{x \in \Omega: V(x) \le 1 \}$ is an invariant set.

{\textbf{\underline{Proof $A^* \cap \{x \in \Omega: V(x) \le 1 \}$ has an attracting} \underline{neighborhood:}}} We now show that for all $y \in A^* \cap \{x \in \Omega: V(x) \le 1 \}$ there exists $\delta>0$ such that for any $\eps>0$ there exists $T \ge 0$ for which
\begin{align} \label{10}
D(\phi_f(z,t),A^* \cap & \{x \in \Omega: V(x) \le 1 \})< \eps\\ \nonumber
& \text{ for all } z \in B(y, \delta) \text{ and } t \ge T.
\end{align}

Let $y \in A^* \cap \{x \in \Omega: V(x) \le 1 \}$ then $y \in A^*$. Since $A^*$ is an attractor set there exists $\delta>0$ such that for any $0<\eps<\sigma$ there exists $T_1 \ge 0$ for which
\begin{align} \label{close to A}
D(\phi_f(z,t),A^*)< \eps \text{ for all } z \in B(y, \delta) \text{ and } t \ge T_1.
\end{align}
Since $0<\eps<\sigma$ and $D(\phi_f(z,t),A^*)< \eps$ $\text{for all } (z,t) \in B(y, \delta) \times [T_1, \infty)$, it follows by Eq.~\eqref{subset} that
\begin{align} \label{11}
\phi_f(z,t) \in \Omega \text{ for all } (z,t) \in B(y, \delta) \times [T_1, \infty).
\end{align}

Next, we will consider the cases $\Omega / B(\{x \in \Omega : V(x) \le 1  \},\eps)= \emptyset$ and $\Omega / B(\{x \in \Omega : V(x) \le 1  \},\eps) \ne \emptyset$ separately showing Eq.~\eqref{10} holds for each case.

In the case $\Omega / B(\{x \in \Omega : V(x) \le 1  \},\eps)= \emptyset$ we get that $\Omega \subseteq B(\{x \in \Omega : V(x) \le 1  \},\eps)$, and hence using this fact together with Eq.~\eqref{11} it follows that,
\begin{align} \label{5}
D(\phi_f(z,t),\{x \in \Omega :& V(x) \le 1  \})< \eps \\ \nonumber
&\text{ for all } (z,t) \in B(y, \delta) \times [T_1, \infty ).
\end{align}
Now, Eqs.~\eqref{close to A} and~\eqref{5} imply
\begin{align} \label{8}
&D(\phi_f(z,t),A^* \cap \{x \in \Omega: V(x) \le 1 \}) \\ \nonumber
&\le  \max\{ D(\phi_f(z,t),A^*), D(\phi_f(z,t),\{x \in \Omega : V(x) \le 1  \}) \}< \eps \\ \nonumber
& \hspace{2cm} \text{ for all } (z,t) \in B(y, \delta) \times [T_1, \infty ).
\end{align}
Thus Eq.~\eqref{8} shows Eq.~\eqref{10} in the case $\Omega / B(\{x \in \Omega : V(x) \le 1  \},\eps)= \emptyset$.

Next let us consider the case $\Omega / B(\{x \in \Omega : V(x) \le 1  \},\eps) \ne \emptyset$. By Eqs.~\eqref{Lf} and~\eqref{11}, Gronwall's inequality (Lem.~\ref{lem: gronwall}), and the semi-group property of solution maps (Eq.~\eqref{soln map property}) we have that
\begin{align} \label{4}
V(\phi_f(z,T_1+t)) - 1 \le e^{- t} & ( V(\phi_f(z,T_1)) - 1) \le a e^{- t}\\ \nonumber
&   \text{ for all } (z,t) \in B(y, \delta) \times [0,\infty),
\end{align}
where $a:= \sup_{z \in B(y, \delta)} |V(\phi_f(z,T_1))-1| \ge 0$. Hence, it now follows for any $\eta>0$ that
\begin{align} \label{6}
\phi_f(z,t+T_1) & \in \{x \in \Omega : V(x) \le 1 + \eta \} \\ \nonumber
&\text{ for all } z \in B(y, \delta) \text{ and } t \ge \max \left\{0, \ln \left( \frac{a}{\eta} \right) \right\}. \end{align}

Let $T_2:=T_1+ \max \left\{0, \ln \left( \frac{a}{\eta} \right) \right\}$. We now construct $\eta>0$ such that $\{x \in \Omega : V(x) \le 1 + \eta \} \subseteq B(\{x \in \Omega : V(x) \le 1  \}, \eps)$. Then since $\phi_f(y,t) \in\{x \in \Omega : V(x) \le 1 + \eta \}$ for all $t \ge T_2$ (by Eq.~\eqref{6}), it follows that $D(\{x \in \Omega : V(x) \le 1  \},\phi_f(y,t) ) < \eps$ for all $t \ge T_2$.

Let $\eta \in  (0, b)$ where $b:=\inf_{ z \in \Omega / B( \{x \in \Omega : V(x) \le 1  \} ,\eps) }(V(z) - 1) $, where $\inf_{ z \in \Omega / B( \{x \in \Omega : V(x) \le 1  \} ,\eps) }V(z)$ exists since $\Omega / B(\{x \in \Omega : V(x) \le 1  \},\eps)$ is compact and $V$ is continuous. Note that $b>0$ since $\inf_{ z \in \Omega / B( \{x \in \Omega : V(x) \le 1  \} ,\eps) }V(z) - 1>0$ (because $\Omega / B(\{x \in \Omega : V(x) \le 1  \},\eps)$ is compact so by the extreme value theorem there exists $z^* \in \Omega / B(\{x \in \Omega : V(x) \le 1  \},\eps)$ such that $V(z^*)=\inf_{ z \in \Omega / B( \{x \in \Omega : V(x) \le 1  \},\eps) }V(z)$ and since $z^* \notin \{x \in \Omega : V(x) \le 1  \}$ it follows that $V(z^*)> 1$).

We now claim that $\{x \in \Omega : V(x) \le 1 + \eta \} \subseteq B(\{x \in \Omega : V(x) \le 1  \}, \eps)$. Suppose for contradiction that $\{x \in \Omega : V(x) \le 1 + \eta \} \nsubseteq B(\{x \in \Omega : V(x) \le 1  \}, \eps)$. Then there exists $w \in \{x \in \Omega : V(x) \le 1 + \eta \} \subseteq \Omega$ such that $w \notin B(\{x \in \Omega : V(x) \le 1  \}, \eps)$ implying $w \in \Omega / B( \{x \in \Omega : V(x) \le 1  \} , \eps)$. Now, $V(w) \le \eta + 1 < \inf_{ z \in \Omega / B(\{x \in \Omega : V(x) \le 1  \},\eps) }\{V(z)\} \le V(w)$ implying $0<0$, providing a contradiction. 

Therefore, taking $t \ge T_2$ it follows from Eq.~\eqref{6} that $\phi_f(z,t) \in \{x \in \Omega : V(x) \le 1 + \eta \} \subseteq B(\{x \in \Omega : V(x) \le 1  \}, \eps)$ for all $(z,t) \in B(y, \delta) \times [T_2, \infty )$ implying,
\begin{align} \label{7}
D(\phi_f(z,t),\{x \in \Omega :& V(x) \le 1  \})< \eps \\ \nonumber
&\text{ for all } (z,t) \in B(y, \delta) \times [T_2, \infty ).
\end{align}
Now, Eqs.~\eqref{close to A} and~\eqref{7} it follows that
\begin{align} \label{9}
&D(\phi_f(z,t),A^* \cap \{x \in \Omega: V(x) \le 1 \}) \\ \nonumber
&\le  \max\{ D(\phi_f(z,t),A^*), D(\phi_f(z,t),\{x \in \Omega : V(x) \le 1  \}) \}< \eps \\ \nonumber
& \hspace{2cm} \text{ for all } (z,t) \in B(y, \delta) \times [T_2, \infty ).
\end{align}
Therefore Eqs.~\eqref{8} and~\eqref{9} prove Eq.~\eqref{10}.

\end{proof}
 We now propose an SOS optimization problem for enforcing the constraints of Optimization Problem~\eqref{opt: approx minimal A}.

\subsection{An SOS Representation of the Lyapunov Inequality Constraint} \label{subsec: SOS constraints}
Suppose $A^* \subset \R^n$ is the minimal attractor of some ODE~\eqref{eqn: ODE} (defined by the vector field $f:\R^n \to \R^n$) and $\Omega \subset BOA_f(A^*)$ is some compact set such that $B(A^*,\sigma) \subseteq \Omega^\circ$, for some $\sigma>0$. Let us consider the problem of approximating the minimal attractor $A^*$ by some set $A$ that can be certified as an attractor set (but not necessarily the minimal attractor set). One way to approach this problem is by solving Opt.~\eqref{opt: approx minimal A}, since any feasible solution, $J$, to Opt.~\eqref{opt: approx minimal A} satisfies the Lyapunov conditions of Prop.~\ref{prop: LF implies attractor set}, and hence, $A:=\{x \in \Omega: J(x) \le 1 \}$ is an attractor set to the ODE defined by a vector field, $f$.

We now consider how to enforce the conditions of Opt.~\eqref{opt: approx minimal A} using SOS optimization. Fortunately it is not necessary to enforce the constraint $\{x \in \Omega: J(x) \le 1\} \ne \emptyset$ since Lem.~\ref{lem: S0 is nonempty} shows that $J$ automatically satisfies this constraint when $A^* \subseteq \Omega$. We next propose a SOS tightening of the remaining constraints of Opt.~\eqref{opt: approx minimal A}, taking $\Omega$ to have the form $\Omega= \{ x \in \R^n: g_\Omega(x) \ge 0\}$ {with} $\partial \Omega = \{ x \in \R^n: g_\Omega(x) = 0\}$.

For some $\alpha>0$ we now consider the following optimization problem,
\begin{align} \label{opt: approx minimal A 2}
&\inf_{J \in  \sum^d_{SOS}} D_V(A^*,\{x \in \Omega: J(x) \le 1 \})\\ \nonumber
& \text{such that }   J,s_0,k_0, k_1 \in \sum^d_{SOS}, \quad p_0 \in {\mcl P_d(\R^n, \R)},  \\ \nonumber
& \text{where } k_0(x)=-\nabla J(x) ^T f(x) -  (J(x) - 1) - s_0(x) g_\Omega(x),\\ \nonumber
& k_1(x) = (J(x)-1 - \alpha) - p_0(x) g_\Omega(x).
\end{align}


The problem with solving Opt.~\eqref{opt: approx minimal A 2} in its current form is that evaluating the objective function requires knowledge of the minimal attractor set, $A^*$ (which is unknown). Fortunately, however, we can formulate an optimization problem which is equivalent to Opt.~\eqref{opt: approx minimal A 2}, but with an objective function that does not depend on the unknown minimal attractor set, $A^*$.

If $B(A^*,\sigma) \subseteq \Omega^\circ$, for some $\sigma>0$ and $J$ is feasible to Opt.~\eqref{opt: approx minimal A 2}, then Corollary~\ref{cor: convergence of SOS to attrcator} shows that minimizing $D_V(A^*,\{x \in \Omega: J(x) \le 1 \})$ is equivalent to minimizing $\mu(\{x \in \Omega: J(x) \le 1 \})$. Roughly speaking, if $J$ is feasible to Opt.~\eqref{opt: approx minimal A 2} then $J$ satisfies the constraints of Opt.~\eqref{opt: approx minimal A 2}. Hence $\nabla J(x)^Tf(x) \le -(\nabla J(x) -1) \text{ for all } x \in \Omega.$ Thus if $B(A^*,\sigma) \subseteq \Omega^\circ$, where $\sigma>0$ and $A^*$ is the minimal attractor of the ODE defined by $f$, it follows by Lemma~\ref{lem: S0 contains minimal attractor} that $A^* \subseteq \{x \in \Omega: J(x) \le 1 \}$. Hence, by Lem.~\ref{lem: D_V is related to vol} we have that $D_V(A^*,\{x \in \Omega: J(x) \le 1 \})= \mu(\{x \in \Omega: J(x) \le 1 \})- \mu(A^*)$. Now, $\mu(A^*)$ is a constant (since $A^*$ is not a decision variable). Therefore minimizing $D_V(A^*,\{x \in \Omega: J(x) \le 1 \})$ is equivalent to minimizing $\mu(\{x \in \Omega: J(x) \le 1 \})$.

For some $\alpha>0$, we now consider the following family of $d$-degree SOS problems,
\begin{align} \label{opt: SOS tightening}
J_{d, \alpha} & \in  \arg \inf_{J}\mu(\{x \in \Omega:J(x) \le 1 \})\\ \nonumber
&  J,s_0,k_0,k_1 \in \sum^d_{SOS}, \quad p_0 \in \mcl P_d(\R^n, \R)  \\ \nonumber
& \text{where } k_0(x)=-\nabla J(x) ^T f(x) -  (J(x) - 1) - s_0(x) g_\Omega(x) \\ \nonumber
& k_1(x) = (J(x)-1 - \alpha) - p_0(x) g_\Omega(x).
\end{align}

We now show that for sufficiently small $\alpha>0$ and ``large" $\Omega$ our quasi-SOS optimization problem proposed in Opt.~\eqref{opt: SOS tightening} is not conservative since for sufficiently large enough degree its solution yields an arbitrarily close approximation of the minimal attractor set (in the volume metric). Moreover, each solution to  Opt.~\eqref{opt: SOS tightening} yields an attractor set.

\begin{cor} \label{cor: convergence of SOS to attrcator}
	Consider $f \in \mcl P(\R^n, \R)$. Suppose $A^* \subset \R^n$ is a minimal attractor set to the ODE~\eqref{eqn: ODE} defined by $f$, $\sigma>0$, and $\Omega \subset \R^n$ is some compact set such that $B(A^*,\sigma) \subseteq \Omega$ and $\Omega \subset BOA_f(A^*)$, $\Omega = \{x\in \R^n: g_\Omega(x) \ge 0 \}$, and $\partial \Omega = \{x\in \R^n: g_\Omega(x) = 0 \}$, where $g_\Omega \in \mcl P(\R^n , \R)$. Suppose $\{J_{d, \alpha}\}_{d \in \N}$ is such that $J_{d, \alpha}$ solves the $d$-degree optimization problems given in Eq.~\eqref{opt: SOS tightening} for $\alpha>0$, then:
	\begin{enumerate}
		\item $\{x \in \Omega:J_{d, \alpha}(x) \le 1 \}$ is an attractor set for each $d \in \N$ and $\alpha>0$.
				\item $A^* \subseteq \{x \in \Omega:J_{d, \alpha}(x) \le 1 \}$ for each $d \in \N$ and $\alpha>0$.
		\item There exists $\beta>0$ such that for any $\alpha \in (0,\beta)$ we have that $\lim_{d \to \infty} D_V(A^*, \{x \in \Omega:J_{d, \alpha}(x) \le 1 \})=0$.
	\end{enumerate}
\end{cor}
\begin{proof}
	In order to prove  Cor.~\ref{cor: convergence of SOS to attrcator} we will now split the remainder of the proof into three parts showing each of the three statements of Cor.~\ref{cor: convergence of SOS to attrcator}.
	
	\underline{\textbf{Proof $\{x \in \Omega:J_{d, \alpha}(x) \le 1 \}$ is an attractor set:}} By Prop.~\ref{prop: LF implies attractor set} it follows that $\{x \in \Omega:J_{d, \alpha}(x) \le 1 \}$ is an attractor set if $J_{d, \alpha}$ satisfies Eqs.~\eqref{eqn: LF ineq}, \eqref{eqn: S0 inside intererior} and~\eqref{eqn: S0 non-empty}. Since $J_{d, \alpha}$ is assumed to feasible to Opt.~\eqref{opt: SOS tightening} it follows that $J_{d, \alpha}$ satisfies the constraints of Opt.~\eqref{opt: SOS tightening} and hence $J_{d, \alpha}$ trivially satisfies Eq.~\eqref{eqn: LF ineq}. Moreover by the constraints of Opt.~\eqref{opt: SOS tightening} it follows that $J_{d, \alpha}(x) \ge 1+ \alpha>1$ for all $x \in \partial \Omega$ and hence $\{x \in \Omega: J_{d, \alpha}(x) \le 1\} \subseteq \Omega^\circ$, implying $J_{d, \alpha}$ satisfies Eq.~\eqref{eqn: S0 inside intererior}. Finally since $A^* \subseteq \Omega$ and $J_{d, \alpha}$ satisfies Eq.~\eqref{eqn: LF ineq} it follows by Lem.~\ref{lem: S0 is nonempty} that $J_{d, \alpha}$ satisfies Eq.~\eqref{eqn: S0 non-empty}.
	
\underline{\textbf{Proof $A^* \subseteq \{x \in \Omega:J_{d, \alpha}(x) \le 1 \}$:}} Since $J_{d, \alpha}$ satisfies the constraints of Opt.~\eqref{opt: SOS tightening} it follows that
\begin{align*}
\nabla J_{d, \alpha}(x)^Tf(x) \le -(\nabla J_{d, \alpha}(x) -1) \text{ for all } x \in \Omega.
\end{align*}
Since $B(A^*,\sigma) \subset \Omega$ and $A^*$ is the minimal attractor set of the ODE defined by $f$ it follows from Lemma~\ref{lem: S0 contains minimal attractor} that $A^* \subseteq \{x \in \Omega:J_{d, \alpha}(x) \le 1 \}$.

	\underline{\textbf{Proof $\lim_{d \to \infty} D_V(A^*, \{x \in \Omega:J_{d, \alpha}(x) \le 1 \})=0$:}} We show that there exists $\beta>0$ such that for any $\alpha\in (0,\beta)$ and $\eps>0$ there exists $N \in \N$ such that $D_V(A^*, \{x \in \Omega:J_{d, \alpha}(x) \le 1 \})<\eps$ for all $d \ge N$.
	
 By Theorem~\ref{thm: existence of sos LF} it follows that there exists $\beta>0$ such that for $\eps>0$ there exists $N_1 \in \N$, and $P_m \in \sum_{SOS}^m(\R^n, \R)$ such that
 \begin{align} \label{Jm satisfies LF ineq}
&\nabla P_m(x)^Tf(x)<-(P_m(x)-1) \text{ for all } x \in \Omega \text{ and } m >N_1,\\ \label{Jm greater on boundary}
&P_m(x)>1 + \beta > 1 + \alpha\\ \nonumber
& \hspace{2cm} \text{ for all } x \in \partial \Omega, \alpha \in (0,\beta), \text{ and } m >N_1,\\ \label{Jm contains minimal attractor}
&A^* \subseteq \{x \in \Omega:P_m(x) \le 1 \} \text{ for all } m > N_1,\\ \label{Jm approx A}
& D_V(A^*,\{x \in \Omega:P_m(x) \le 1 \} )< \eps \text{ for all } m \ge N_1.
 \end{align}

For any $\alpha \in (0,\beta)$,  by Eqs.~\eqref{Jm satisfies LF ineq} and~\eqref{Jm greater on boundary} and Theorem~\ref{thm: Psatz} there exists $s_0,s_1,s_2,s_3,s_4,s_5 \in \sum_{SOS}$ for each $m > N_1$ such that $-\nabla P_m(x)^Tf(x) -(P_m(x)-1) -s_0(x) g_\Omega(x) =s_1(x)$, and $(P_m(x) - 1 - \alpha) - s_2(x)g_\Omega(x)= s_3(x)$, and $(P_m(x) - 1 - \alpha) + s_4(x)g_\Omega(x)= s_5(x)$. Fix $m > N_1$ and let $N_2:= \max\{m, \max_{0 \le i \le 5}\deg(s_i)  \}$. Then it follows that $P_m$ is feasible to Opt.~\eqref{opt: SOS tightening} for degree $d \ge N_2$ (with $p_0(x):=0.5(s_4(x) -s_2(x))$). Since, $J_{d, \alpha}$ solves the Opt.~\eqref{opt: SOS tightening} and $P_m$ is feasible to Opt.~\eqref{opt: SOS tightening} it follows that,
	\begin{align} \label{Pd less than Jm}
 \mu(\{x \in \Omega:J_{d,\alpha}(x) \le 1 \}) \le & \mu(\{x \in \Omega:P_m(x) \le 1 \}) \\ \nonumber
& \text{ for all } d \ge N_2.
	\end{align}

	Hence, using Lemma~\ref{lem: D_V is related to vol} along with the fact that $A^* \subseteq \{x \in \Omega:J_{d,\alpha}(x) \le 1 \}$ (by Lem.~\ref{lem: S0 contains minimal attractor}) and Eqs.~\eqref{Jm contains minimal attractor},~\eqref{Jm approx A}, and~\eqref{Pd less than Jm}, it follows that,
	\begin{align*}
	  D_V(A^*, & \{x \in \Omega:J_{d,\alpha}(x) \le 1 \})  \\
	 & = \mu(\{x \in \Omega:J_{d,\alpha}(x) \le 1 \}) - \mu(A^*)\\
	 & \le \mu(\{x \in \Omega:P_m(x) \le 1 \}) - \mu(A^*)\\
	 & =  D_V(A^*, \{x \in \Omega:P_m(x) \le 1 \})\\
	 &<\eps \text{ for all } d \ge N_2.
	\end{align*} \end{proof}

\subsection{Heuristic Volume Minimization of Sublevel Sets of SOS Polynomials} \label{subsec: det obj}

Unfortunately, it is still not possible for us to solve the family of $d$-degree optimization problems given in Eq.~\eqref{opt: SOS tightening} since there is no known convex closed form analytical expression for the objective function (the volume of a sublevel set of an SOS polynomial). To make the problem tractable we replace the objective function in Eq.~\eqref{opt: SOS tightening} with a convex objective function based on the determinant. We next present two convex candidate objective functions based on the determinant.
{\begin{rem}
The functions $f_1: S^n_{++} \to \R$ and $f_2: S^n_{++} \to \R$ defined as,
\begin{align*}
f_1({P})&= -\log \det(P),\\
f_2({P})&=- (\det(P))^{\frac{1}{n}},
\end{align*}
are convex.
\end{rem}}
Heuristically, maximizing $\det(P)$ increases the value of $V(x)=z_d(x)^T P z_d(x)$ for all $x \in \R^n$. Therefore, for larger $\det(P)$ there will be less $y \in \R^n$ such that $y \in \{x\in \R^n: V(x) \le 1 \}$. Hence we would expect $\mu( \{x\in \R^n: V(x) \le 1 \})$ to decrease as $\det(P)$ increases. In the $2$-degree (quadratic) case this argument is not heuristic. We next show that maximizing the determinant is equivalent to minimizing the volume of the sublevel set of a quadratic polynomial.
\begin{lem}[\cite{jones2018using}] \label{lem: vol of ellipse and det}
	Consider $P \in S^n_{++}$. The following holds,
	\begin{align*}
	\mu(\{x\in \R^n: x^T P x \le 1 \})= \frac{\pi^\frac{n}{2} }{\Gamma(\frac{n}{2}+1) \sqrt{\det(P)} },
	\end{align*}
	where $\Gamma$ is the gamma function.
\end{lem}
Lemma~\ref{lem: vol of ellipse and det} shows that maximizing $\det(P)$ minimizes $\mu(\{x\in \R^n: x^T P x \le 1 \})$. Thus equivalently, maximizing the convex functions $\log \det(P)$ or $ (\det(P))^{\frac{1}{\mcl N_d}}$ minimizes $\mu(\{x\in \R^n: x^T P x \le 1 \})$ (since both the functions $f_1(x)= \log(x)$ and $f_2(x) = x^{\frac{1}{n}}$ are monotonic functions for $x >0$). We next extend this approach of maximizing the determinant to minimize the volume of a sublevel set of a SOS polynomial to higher degrees.

Rather than solving Opt.~\eqref{opt: SOS tightening} we solve the following family $d$-degree SOS problems for some $\alpha>0$,
\begin{align} \label{opt: SOS det}
P_d & \in  \arg \sup_{J}(\det P)^{\frac{1}{\mcl N_d}}\\ \nonumber
&  J,s_0,k_0,k_1 \in \sum^d_{SOS},  \quad p_0 \in \mcl P_d(\R^n, \R)  \\ \nonumber
&  \text{where, } J=z_d(x)^T P z_d(x), \text{ and,}\\ \nonumber
& P>0,\\ \nonumber
&k_0(x)=-\nabla J(x) ^T f(x) -  (J(x) - 1) - s_0(x) g_\Omega(x)\\ \nonumber
& k_1(x) = (J(x)-1 - \alpha) - p_0(x) g_\Omega(x).
\end{align}
Note that it is equivalent to solve Opt.~\eqref{opt: SOS det} with an objective function of form $(\det P)^{\frac{1}{\mcl N_d}}$ or $\log \det P$. For implementation purposes we have chosen to use an objective function of form $(\det P)^{\frac{1}{\mcl N_d}}$ since Yalmip~\cite{lofberg2004yalmip} allows this formulation of the problem to be solved by various SDP solves. For an objective function of form $\log \det P$ we use SOSTOOLS~\cite{prajna2002introducing} and SDPT3~\cite{sdpt3_2003}.

\subsection{A Further Simplification of Optimization Problem~\eqref{opt: SOS det}} \label{subsec:futher simp} Typically, through numerical experimentation, we find that if sublevel set volume of $\{x \in \Omega : J(x) \le 1\}$ is sufficiently minimized and $\Omega$ is sufficiently large, then $\{x \in \Omega : J(x) \le 1\}  \subseteq \Omega^\circ$ is automatically satisfied. Therefore, it is often unnecessary to enforce the constraint $(J(x)-1 - \alpha) - p_0(x) g_\Omega(x) \in \sum_{SOS}^d$ in Opt.~\eqref{opt: SOS det} -- thus further reducing computational complexity of the SOS programming problem.

\section{Numerical Examples} \label{sec: numerical examples}
In this section we will present the results of solving
the Opt.~\eqref{opt: SOS det} for several dynamical systems. For
these examples Opt.~\eqref{opt: SOS det} was solved using Yalmip~\cite{lofberg2004yalmip}. In Example~\ref{ex:lorenz} we approximate a ``strange" attractor, in Example~\ref{ex: van} we approximate a limit cycle, and finally in Example~\ref{ex: ahmadi} we approximate an equilibrium point.

\begin{ex}[Numerical Approximation of the Lorenz Attractor] \label{ex:lorenz}

Consider the following three dimensional second order nonlinear dynamical system (known as the Lorenz system):
\begin{align} \label{eqn: Lorenz attactor}
\dot{x}_1(t) & = \sigma ( x_2(t) - x_1(t) ), \\ \nonumber
\dot{x}_2(t) & = \rho x_1(t) - x_2(t) - x_1(t)x_3(t), \\ \nonumber
\dot{x}_3(t) & = x_1(t)x_2(t) - \beta x_3(t),
\end{align}
where $(\sigma,\rho,\beta) = (10,28,\frac{8}{3})$. It is well known that for such $(\sigma,\rho,\beta)$ the ODE~\eqref{eqn: Lorenz attactor} exhibits a global ``chaotic" attractor.

Fig.~\ref{fig:Lorentz_outer_approx} shows our Lorenz attractor approximation given by the $1$-sublevel of the solution to the SOS Problem~\eqref{opt: SOS det} for $d=8$, $\alpha=0.0001$, $g_\Omega(x)=R^2 - x_1^2 - x_2^2 - x_3^2$, $R=3$, and scaled dynamics given by the ODE~\eqref{eqn: Lorenz attactor}. For $d \ge 10$ the volume of our Lorenz attractor approximation becomes so small that we are unable store enough grid-points to sufficiently plot the contour of the $1$-sublevel set of our SOS Lyapunov function.

\begin{figure}
	\flushleft
	\includegraphics[scale=0.6]{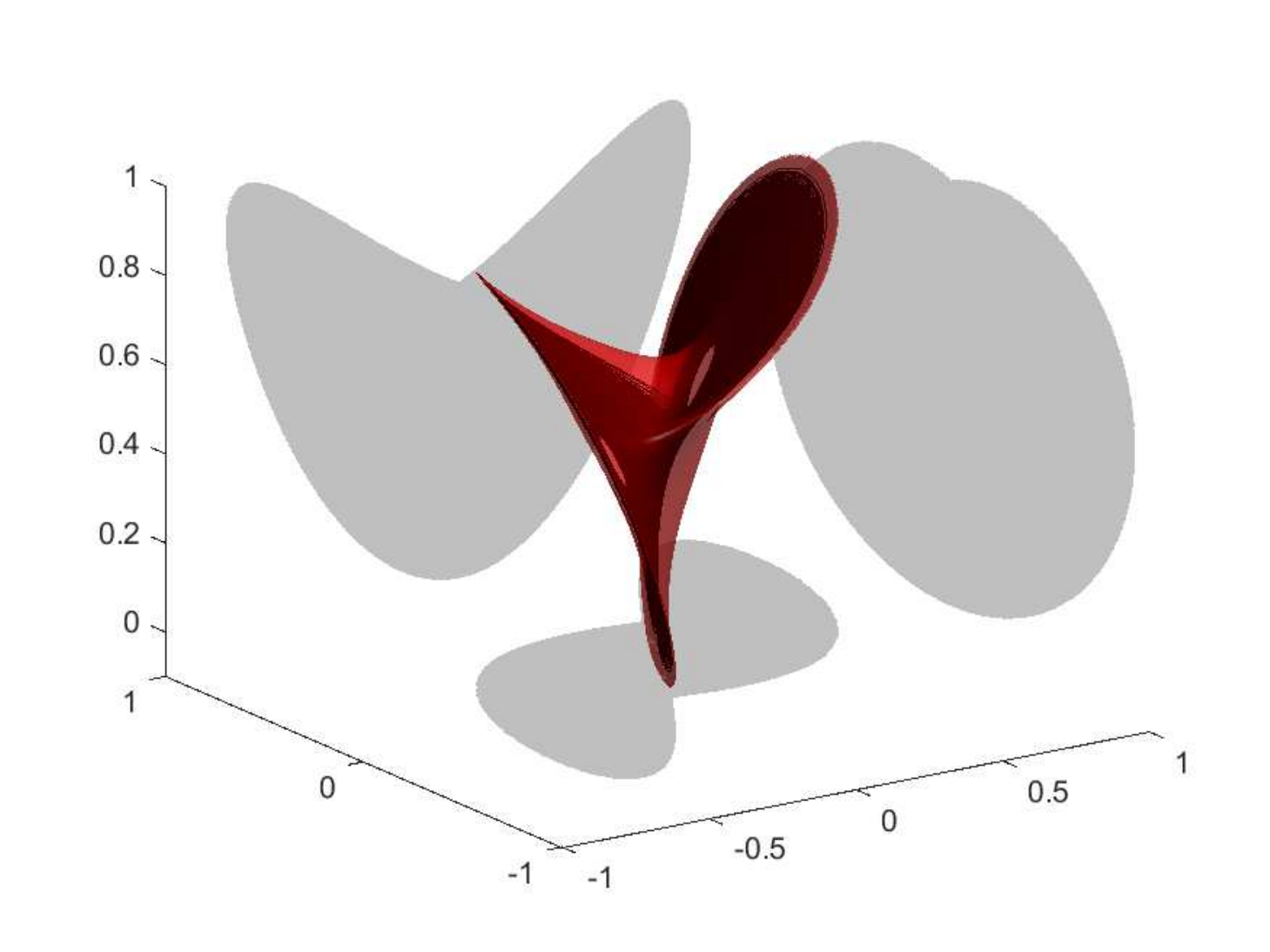}
	\vspace{-20pt}
	\caption{Graph showing an estimation of the Lorenz attractor (Example~\ref{ex:lorenz}) given by the red transparent surface. This surface is the $1$-sublevel set of a solution to the SOS Problem~\eqref{opt: SOS det}. The grayed shaded surfaces represent the projection of the our Lorenz attractor estimation on the xy, xz, and yz axes. The black line is an approximation of the attractor found by simulating a Lorenz trajectory using Matlab's $\texttt{ODE45}$ function. }\label{fig:Lorentz_outer_approx}
	\vspace{-10pt}
\end{figure}

\end{ex}

\begin{ex}[Numerical approximation of the Van der Poll oscillator] \label{ex: van}
	Consider the following two dimensional third order nonlinear dynamical system:
	\begin{align} \label{eqn: van}
	\dot{x_1}(t)&=x_2(t), \\ \nonumber
	\dot{x_2}(t)& = (1-x_1^2(t))x_2(t) - x_1(t).
	\end{align}	
	It is well known that the ODE~\eqref{eqn: ODE} possess a limit cycle called the Van der Poll oscillator. Let us denote this limit cycle by $A^* \subset \R^n$. The ODE also possess an unstable equilibrium point at the origin (which is not an attractor set since the solution map initialized inside neighborhoods of the origin moves away from the origin towards the limit cycle). However, $\phi_f(0,t) =0 \in \R^2$ for all $t \ge 0$, where $\phi_f$ is the solution map of the ODE~\eqref{eqn: ODE}. Therefore, $0 \notin BOA_f(A^*)$. In order to apply Theorem~\ref{thm: existence of sos LF} we require $\Omega \subset BOA_f(A^*)$. Hence, we must be careful to construct  $\Omega = \{x\in \R^n: g_\Omega(x) \ge 0 \}$ such that $0 \notin \Omega$.
	
	Fig.~\ref{fig:van} shows our Van der Poll oscillator approximation given by the $1$-sublevel of the solution to the SOS Problem~\eqref{opt: SOS det} for $d=12$, $\alpha=0.0001$, {$g_1(x)=-(R_1^2 - x_1^2 - x_2^2)( R_2^2 - x_2^2 - x_2^2)$, $R_1=0.45$, $R_2=1$}, and scaled dynamics given by the ODE~\eqref{eqn: van}.

	\begin{figure}
		\flushleft
		\includegraphics[scale=0.6]{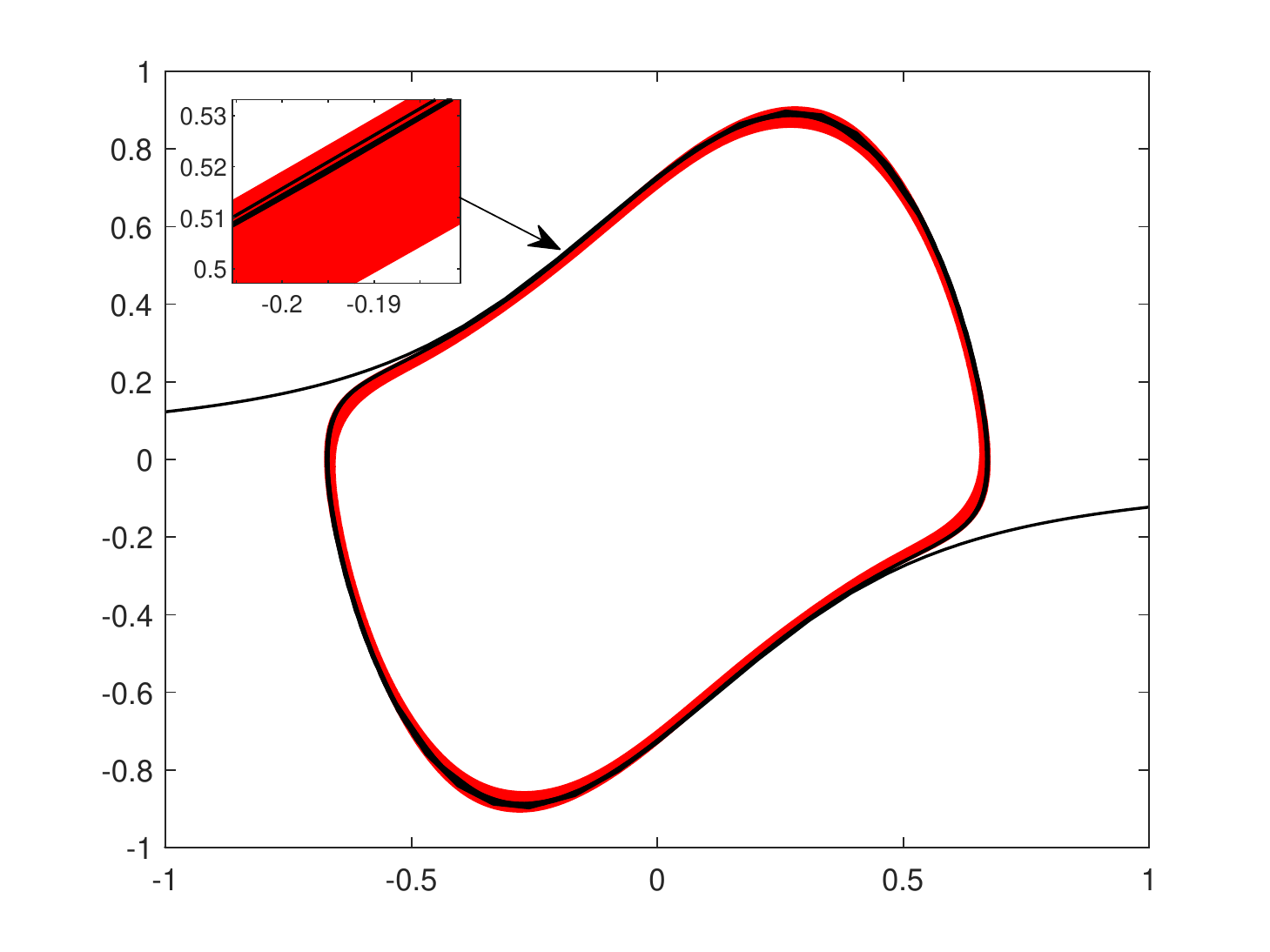}
		\vspace{-20pt}
		\caption{Graph showing an estimation of the attractor (given by the red area) of the ODE~\eqref{eqn: van} in Example~\ref{ex: van}. This red area is the $1$-sublevel set of a solution to the SOS Problem~\eqref{opt: SOS det}. The two black lines are simulated solution maps of the ODE~\eqref{eqn: no poly LF} using Matlab's $\texttt{ODE45}$ function initialized outside of the limit cycle. }\label{fig:van}
		\vspace{-10pt}
	\end{figure}
	
\end{ex}

\begin{ex} \label{ex: ahmadi}
Consider the following two dimensional seventh order nonlinear dynamical system:
\begin{align} \label{eqn: no poly LF}
\dot{x_1}(t)&=-2x_2(t)(-x_1^4(t) + 2x_1^2(t)x_2^2(t) + x_2^4(t)) \\ \nonumber
& \quad  - 2x_1(t)(x_1^2(t) + x_2^2(t))(x_1^4(t) + 2x_1^2(t)x_2^2(t) - x_2^4(t)),\\ \nonumber
\dot{x_2}(t)& = 2x_1(t)(x_1^4(t) + 2x_1^2(t)x_2^2(t) - x_2^4(t)) \\ \nonumber
& \quad  - 2x_2(t)(x_1^2(t) + x_2^2(t))(-x_1^4(t) + 2x_1^2(t)x_2^2(t) + x_2^4(t)).
\end{align}	
It was shown in~\cite{ahmadi2018globally} that $A^*=\{0\}$ is a global attractor set of the ODE~\eqref{eqn: no poly LF}. In other words, the ODE~\eqref{eqn: no poly LF} is globally asymptotically stable about the origin. This stability was shown using the following non-polynomial Lyapunov function:
\begin{align*}
W(x) = \begin{cases}
\frac{x_1^4 + x_2^4}{x_1^2 + x_2^2} \text{ if } x \ne 0\\
0 \text{ otherwise}
\end{cases}.
\end{align*}
Clearly, $W$ is not a SOS polynomial (or even polynomial). It was further shown in~\cite{ahmadi2018globally} that there exists no polynomial Lyapunov function that can certify the asymptotic stability of the origin of the ODE~\eqref{eqn: no poly LF}. However, Theorem~\ref{thm: existence of sos LF} implies there does exist a SOS Lyapunov functions that can certify the stability of an arbitrarily small neighbourhood of $A^*=\{0\}$ with respect the volume metric. Furthermore, we can heuristically attempt to find these Lyapunov functions by solving the SOS Opt.~\eqref{opt: SOS det}.

Fig.~\ref{fig:ahmadi} shows our approximation of the ODE~\eqref{eqn: no poly LF} given by the $1$-sublevel of the solution to the SOS Problem~\eqref{opt: SOS det} for $d=10$, $\alpha=0.0001$, $g_\Omega(x)=R^2 - x_1^2 - x_2^2$, $R=1$, and $f$ as in the ODE~\eqref{eqn: no poly LF} (scaled by a factor of 1000 to improve SDP solver performance). Unfortunately, increasing $d \in \N$ to a greater value than $10$ makes the SDP solver (Mosek) return a numerical error. We believe improvements in SDP solvers for large scale problems will allow us to solve the SOS Opt.~\eqref{opt: SOS det} for larger degrees and improve our estimations of attractor sets. {Furthermore, we note that this lack of convergence as $d \to \infty$ may demonstrate the limitations in our approach of using matrix determinant maximization to minimize the volume of sublevel sets.}

\begin{figure}
	\flushleft
	\includegraphics[scale=0.6]{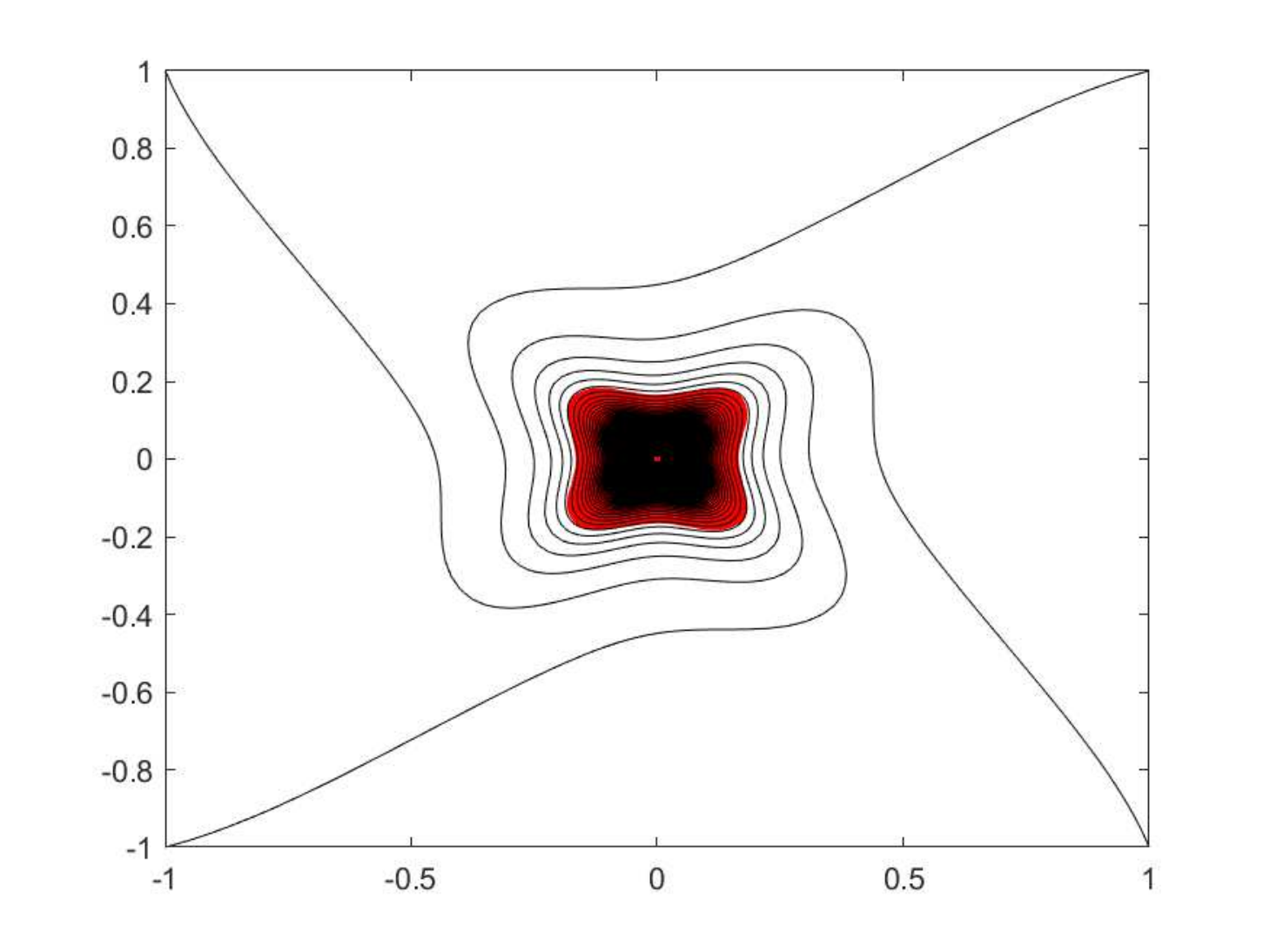}
	\vspace{-20pt}
	\caption{Graph showing an estimation of the attractor (given by the red area) of the ODE~\eqref{eqn: no poly LF} in Example~\ref{ex: ahmadi}. This red area is the $1$-sublevel set of a solution to the SOS Problem~\eqref{opt: SOS det}. The four black lines are simulated solution maps initialized at $[ \pm 1, \pm 1]^T$ of the ODE~\eqref{eqn: no poly LF} using Matlab's $\texttt{ODE45}$ function. }\label{fig:ahmadi}
	\vspace{-10pt}
\end{figure}

\end{ex}

%
%

\section{Conclusion} \label{sec: conclusion}
We have proposed a new Lyapunov characterization of attractor sets that is well suited to the problem of finding the minimal attractor set. We have shown that our proposed Lyapunov characterization of attractor sets is non-conservative even when restricted to SOS Lyapunov functions. Specifically, given an attractor set associated with some ODE we have shown that there exists a sequence of SOS Lyapunov functions that yield a sequence of sublevel sets, each containing the attractor set, each being an attractor set themselves, and converging to the attractor set in the volume metric. We have used this theoretical result to design an SOS based algorithm for minimal attractor set approximation based on determinant maximization as a proxy for sublevel set volume minimization. Several numerical examples demonstrate how our proposed SOS based algorithm can provide tight approximations of several well known attractor sets such as the Lorenz attractor and Van-der-Poll oscillator.

\section{Appendix A: Miscellaneous Results}


\begin{lem}[Sublevel sets of continuous functions are closed] \label{lem: sublevel set is closed}
Suppose $f \in C(\R^n, \R)$ and $\Omega$ is compact set, then the set $\{x \in \Omega: f(x) \le c\}$, where $c \in \R$, is closed.
\end{lem}
\begin{proof}
	{Let $\{x_k\}_{k \in \N} \subset \{x \in \Omega: f(x) \le c\}$ be a convergent sequence such that $\lim_{k \to \infty} x_k = x^*$. To prove $\{x \in \Omega: f(x) \le c\}$ is closed we must show that $x^* \in \{x \in \Omega: f(x) \le c\}$. Now, since $\Omega \subset \R^n$ is compact it follows that $\Omega$ is closed and thus since $\{x_k\}_{k \in \N} \subset \Omega$ it follows $x^* \in \Omega$. On the other hand, by continuity it follows that $\lim_{k \to \infty}f(x_k) = f(x^*)$. Since $f(x_k) \le c$ for all $k \in \N$ it must follow that $f(x^*) \le c$. Hence, $x^* \in \{x \in \Omega: f(x) \le c\}$ implying $\{x \in \Omega: f(x) \le c\}$ is a closed set.}
\end{proof}

\begin{lem}[Gronwall's Inequality \cite{Hirch_2004}] \label{lem: gronwall}
	Consider scalars $a,b \in \R$ and functions $u, \beta \in C^1(I,\R)$. Suppose
	\begin{align*}
	\frac{d}{dt}u(t) \le \beta(t) u(t) \text{ for all } t \in (a,b).
	\end{align*}
	Then it follows that
	\begin{align*}
	u(t) \le u(a) \exp\left(\int_a^t \beta(s) ds\right) \text{ for all } t \in [a,b].
	\end{align*}
\end{lem}

%
%
%

\begin{thm}[Polynomial Approximation \cite{peet2009exponentially}] \label{thm:Nachbin}
	Let $E \subset \R^n$ be an open set and $f \in C^1(E, \R)$. For any compact set $K \subseteq E$ and $\eps>0$ there exists  $g \in \mcl P(\R^n, \R)$ such that
	\begin{align*}
	\sup_{x \in K}|D^\alpha f(x) - D^\alpha g(x)| < \eps \text{ for all } |\alpha|\le 1.
	\end{align*}
\end{thm}

\begin{thm}[Putinar's Positivstellesatz \cite{Putinar_1993}] \label{thm: Psatz}
	Consider the semialgebriac set $X = \{x \in \R^n: g_i(x) \ge 0 \text{ for } i=1,...,k\}$. Further suppose $\{x  \in \R^n : g_i(x) \ge 0 \}$ is compact for some $i \in \{1,..,k\}$. If the polynomial $f: \R^n \to \R$ satisfies $f(x)>0$ for all $x \in X$, then there exists SOS polynomials $\{s_i\}_{i \in \{1,..,m\}} \subset \sum_{SOS}$ such that,
	\vspace{-0.4cm}\begin{equation*}
	f - \sum_{i=1}^m s_ig_i \in \sum_{SOS}.
	\end{equation*}
\end{thm}

%

\section{Appendix B: sublevel set approximation} \label{sec: appendix 2}
Recall from Section~\ref{subsec: set metric notation} that
\begin{align} \label{eqn: volume metric}
D_V(A,B):=\mu( (A/B) \cup (B/A) ),
\end{align} where $\mu(A)$ is the Lebesgue measure of $A \subset \R^n$.

 The sublevel approximation results presented in this appendix are required in the proof of Theorem~\ref{thm: existence of sos LF}. 
\begin{defn} \label{def:metric}
	$D: X \times X \to \R$ is a \textit{metric} if the following is satisfied for all $x,y \in X$,
	\vspace{-0.4cm}
	\setlength{\columnsep}{-0.75in}
	\begin{multicols}{2}
		\begin{itemize}
			\item $D(x,y) \ge 0$,
			\item $D(x,y)=0$ iff $x=y$,
			\item $D(x,y)=D(y,x)$,
			\item $D(x,z) \le D(x,y) + D(y,z)$.
		\end{itemize}
	\end{multicols}
\end{defn}
\begin{lem}[\cite{jones2019using}] \label{lem: Dv is metric}
	Consider the quotient space,
	{	\[
		X:= \mcl B \pmod {\{X \subset \R^n : X \ne \emptyset, \mu(X) =0 \}},
		\] } recalling $\mcl B:= \{B \subset \R^n: \mu(B)<\infty\}$ is the set of all bounded sets. Then $D_V: X \times X \to \R$, defined in Eq.~\eqref{eqn: volume metric}, is a metric.
\end{lem}


\begin{lem}[\cite{jones2019using}] \label{lem: D_V is related to vol}
	If $A,B \in \mcl B$ and  $B \subseteq A$ then
	\begin{align*}
	D_V(A,B)& =\mu(A/B)= \mu(A)- \mu (B).
	\end{align*}
\end{lem}
\begin{prop}[\cite{jones2020polynomial}] \label{prop: close in L1 implies close in V norm}
	Consider a set $\Lambda \in {\mcl B}$, a function $V \in L^1(\Lambda, \R)$, and a family of functions $\{J_d \in L^1(\Lambda, \R): d \in \N\}$ that satisfies the following properties:
	\begin{enumerate}
		\item For any $ d \in \N$ we have $J_d(x) \le V(x)$ for all $x \in \Lambda$.
		\item $\lim_{d \to \infty} ||V -J_d||_{L^1(\Lambda, \R)} =0$.
	\end{enumerate}
	Then for all $\gamma \in \R$
	\begin{align} \label{sublevel sets close}
	\lim_{d \to \infty}	D_V \bigg(\{x \in \Lambda : V(x) \le \gamma\}, \{x \in \Lambda : J_d(x) \le \gamma\} \bigg) =0.
	\end{align}
\end{prop}

\bibliographystyle{unsrt}
\bibliography{bib_Approx_sublevel_set}

\end{document}